\documentclass[12pt]{amsart}
\usepackage{amssymb,amsfonts,amsthm, algorithm, algpseudocode, color, enumitem}
\usepackage{graphicx}
\usepackage{bm}
\usepackage{sistyle}
\SIthousandsep{,}
\newtheorem{thm}{Theorem}[section]

\theoremstyle{definition}

\theoremstyle{remark}

\numberwithin{equation}{section}

\newcommand{\olsi}[1]{\,\overline{\!{#1}}} % overline short italic

\begin{document}
\title[Neural network Fokker-Planck solver]{A deep learning method for
solving Fokker-Planck equations}
\author{Jiayu Zhai}
\address{Jiayu Zhai: Department of Mathematics and Statistics, University
  of Massachusetts Amherst, Amherst, MA, 01002, USA}
\email{zhai@math.umass.edu}
\author{Matthew Dobson}
\address{Matthew Dobson: Department of Mathematics and Statistics, University
  of Massachusetts Amherst, Amherst, MA, 01002, USA}
\email{dobson@math.umass.edu}
\author{Yao Li}
\address{Yao Li: Department of Mathematics and Statistics, University
  of Massachusetts Amherst, Amherst, MA, 01002, USA}
\email{yaoli@math.umass.edu}

\thanks{Yao Li is partially supported by NSF DMS-1813246.}


\begin{abstract}
  The time evolution of the probability distribution of a stochastic differential equation
  follows the Fokker-Planck equation, which usually has an unbounded, 
  high-dimensional domain. Inspired by our early study in \cite{li2018data}, we propose a
  mesh-free Fokker-Planck solver, in which the solution to the Fokker-Planck
  equation is now represented by a neural network. The presence of the differential
  operator in the loss function improves the accuracy of the neural
  network representation and reduces the the demand of data in the
  training process. Several high dimensional numerical examples are
  demonstrated. 
\end{abstract}

\maketitle

\maketitle

\section{Introduction}
Stochastic differential equations are widely used to model dynamics
of real world problems in the presence of uncertainty (such as events 
driving stock markets) or in the presence many small forces whose
origins are not all tracked (such as a solvent acting on a larger 
molecule).  The instantaneous and cumulative effects of
the noise on the dynamics can be visualized through
the transient and invariant probability distribution of the solution
process, respectively. These probability measures can be analytically
described by the Fokker-Planck equations (also known as the Kolmogorov
forward equation).   It is well known for the Langevin and Smoluchowski equations
 that if the deterministic part of
the stochastic differential equation is a gradient flow, the invariant measure is the Gibbs
measure whose probability density function is explicitly
given. However in general, the Fokker-Planck equation can only be
solved numerically. Traditional numerical PDE solvers do not work well
for Fokker-Planck equations due to both the lack of a suitable boundary condition and the
curse of dimensionality (see Section \ref{Sec:Problem description} for
detailed explanation). Although many novel methods are introduced to
resolve this difficulty, solving high dimensional Fokker-Planck
equation remains as a challenge.

With the rapid growth of accessibility of data and demand for its
analysis in various application realms, such as computer vision,
speech recognition, natural language processing, and game
intelligence, machine learning methods prove their strong performance
in representing these high dimensional models. Mathematicians have also
made quite many efforts on proving the error estimates of neural
network representations using function spaces like Sobolev spaces,
Besov spaces, and Barron spaces
\cite{Ingo2020Error,Petersen2018Optimal,Taiji2019Adaptivity,E2019Barron}. Although
these results are still far away from explaining their strong
performance, the success of machine learning methods in modelling high
dimensional models in big data applications motivates their consideration
as a numerical scheme for solving mathematics problems, particularly in partial differential
equations. Among others, we highlight the
references \cite{Sirignano2018DGM, Beck2018Kolmogorov,
  Beck2019BackwardSDEs, Raissi2019, wu2018physics} that are related to this
work. 

In \cite{li2018data}, the author 
proposed a novel data-driven solver to solve the Fokker-Planck
equation. The key idea is to remove the reliance on the boundary
conditions and construct a constrained optimization problem that uses
the Monte Carlo simulation data as the reference. This approach is
still grid-based, so it (including its extension in
\cite{dobson2019efficient}) does not work well for high dimensional
problems.  Motivated by the progress of applying artificial neural
network to traditional computational problems, in this paper we
propose a mesh-free version of the data-driven solver studied in
\cite{li2018data, dobson2019efficient}. Similar to those works, the focus here is on the
steady state Fokker-Planck equation as the invariant probability
measure plays a very important role in applications. The case of
time-dependent Fokker-Planck equation is analogous. All our algorithms
can be applied to time-dependent problems with some minor
modifications.

The key idea in this paper is to replace the constrained optimization
problem studied in \cite{li2018data} by an unconstrained optimization
problem, as it is not easy to use neural networks to study the
constrained optimization problem on a high dimensional hyperplane. We
first propose the unconstrained optimization problem and prove the
convergence of its minimizer to the true Fokker-Planck solution for the
discrete case. Then we propose an analogous loss function that is
trainable by artificial neural networks. Our further studies find that
the Fokker-Planck operator $\mathcal{L}$ plays an important role in
the training. It dramatically increases the tolerance of noisy
simulation data and reduces the amount of simulation data used in the
training. In general, we only need $10^{2}$ to $10^{4}$ ``reference
points'' with probability densities on them to train the neural
network. And the probability density function obtained by Monte Carlo simulation does
not have to be very accurate. (see Section
\ref{Sec:Algorithm} for explanation and Section \ref{Sec:Numerical
  experiments} for numerical demonstrations).  The reduction of demand for simulation data is significant since the
stochastic dynamical systems in applications usually have high
dimensionality, whereas the training data collected from either Monte
Carlo simulation or experiments has high cost. This has some
similarity with the situation of the physics-informed neural network 
\cite{Raissi2019, wu2018physics}. In this sense, prior knowledge of the
system \eqref{SDE} or the corresponding Fokker-Planck equation
\eqref{FPE} can serve as a smoother and a law for the solution to
follow. Then we only need a small data set as a reference to locate
the solution near the empirical probability distribution.  

In Section \ref{Sec:Problem description}, we describe the problem
setting, the unconstrained optimization we study, and
the idea of using neural network representation. All training and sampling
algorithms are studied in Section \ref{Sec:Algorithm}. In
Section \ref{Sec:Numerical experiments}, we use several numerical
examples to demonstrate the main feature of our neural network
Fokker-Planck solver.

\section{Preliminaries and motivation}\label{Sec:Problem description}
\subsection{Fokker-Planck equation and data-driven solver.}
\label{Sec: FPE solver}
We consider the stochastic differential equation
\begin{equation}\label{SDE}
d\boldsymbol{X}_{t} = f(\boldsymbol{X}_{t}) dt + \sigma(\boldsymbol{X}_{t}) d\boldsymbol{W}_{t} \,,
\end{equation}
where $f$ is a vector field in $\mathbb{R}^{n}$, $\sigma$ is a
coefficient matrix, and ${\bm W}_{t}$ is an $n$-dimensional white
noise. The time evolution of probability density of the solution process
${\bm X}_{t}$ is characterized by the Fokker-Planck equation, which is also known as the Kolmogorov forward equation 
\begin{equation}\label{FPE}
u_t = \mathcal{L}u = -\sum_{i = 1}^{n} (f_{i}u)_{x_{i}} + \frac{1}{2}\sum_{i,j = 1}^{n}(\Sigma_{i,j}u)_{x_{i} x_{j}} \,,
\end{equation}
where $u(\boldsymbol{x},t)$ denotes the
probability density function of the stochastic process ${\bm X}_t$ at time
$t$, $\Sigma = \sigma^{T}\sigma$ is the diffusion coefficient, and subscripts $t$ and $x_i$ denote partial derivatives. In this paper, we focus on the invariant probability measure of \eqref{SDE}, whose density function $\mathbb{R}^n \ni \boldsymbol{x}\mapsto u(\boldsymbol{x})\in\mathbb{R}$ satisfies the stationary Fokker-Planck equation
\begin{equation}\label{FPE_unique}
\left\{\begin{array}{ll}
\mathcal{L}u = 0\\
\int_{\mathbb{R}^{n}} u \, d\boldsymbol{x} = 1
\end{array}\right.
\end{equation}
Throughout the present paper, we assume the existence and uniqueness
of the solution to the stationary Fokker-Planck equation.

The Fokker-Planck equation is defined on an unbounded domain with the 
constraint $\int_\Omega u\,d\boldsymbol{x} =1$. Since the numerical
domain has to be bounded, it is not easy to give a suitable boundary
condition to describe the ``zero-boundary condition at infinity''. In
practice, one can assume a zero boundary condition on a
domain that is large enough to cover all high density areas with
sufficient margin. A
classic computational method, e.g., finite element method, is then
applied to find a non-trivial solution. One usually needs to solve a
least square problem because of the constraint $\int_{\mathbb{R}^{n}}
u\,d\boldsymbol{x} =1$. In general, the computational cost of
classical PDE solver is too high to be practical when $n \geq 3$. The
other way to solve the Fokker-Planck equation is the Monte Carlo
method, which uses the fact that the empirical distribution of a long
trajectory converges to the solution to the steady state Fokker-Planck
equation. The Monte Carlo method is very simple regardless of the
boundary condition. One only needs to divide the numerical domain into
lots of ``bins'', run a long trajectory of the equation \eqref{SDE},
and count the number of samples in each bin. However, the solution
from the Monte Carlo method is much less accurate. 

In \cite{li2018data}, the author
%corresponding author of this paper 
introduced a data-driven method that
overcomes the drawbacks of the two aforementioned methods, so that one can solve
the Fokker-Planck equation locally and does not rely on the boundary
condition any more. Later in \cite{dobson2019efficient}, the authors proved
the convergence of the method and improved the method by introducing a ``blocked
version'' that uses a divide-and-conquer strategy. Let $D \subset
\mathbb{R}^{n}$ be the numerical domain. Assume there is a rectangular grid
$\{x_{i}\}_{i = 1}^{N^{n}}$ defined in $D$ with $N$ grid points on
each dimension. The key idea of
this data-driven method is to solve the optimization problem 
\begin{equation}\label{opt}
\begin{array}{rl}
\mathop{\min}\limits_{\boldsymbol{u}} & \| \boldsymbol{u} - \boldsymbol{v} \|_{2} \\
\text{subject to} & \boldsymbol{A} \boldsymbol{u} = \boldsymbol{0},
\end{array}
\end{equation}
where $\boldsymbol{A}\in\mathbb{R}^{(N-2)^n\times N^n}$ is a
discretization of the Fokker-Planck operator $\mathcal{L}$ on $D$ without boundary condition,
and $\boldsymbol{v}\in\mathbb{R}^{N^n}$ is a Monte Carlo approximation
obtained by a numerical simulation of \eqref{SDE}. An entry ${\bm
  v}_{i}$ of the vector ${\bm v}$ is the
probability that a long trajectory stays in a small neighborhood of
$x_{i}$, which is usually a low accuracy approximation of the invariant
measure. Each row of matrix ${\bm A}$ is obtained by a discretization of the
Fokker-Planck equation (using the finite difference method) at a
interior point $x_{i}$. Matrix ${\bm A}$ only has $(N-2)^{n}$ rows but $N^{n}$ columns because we
do not know the boundary value. The motivation is that an inaccurate
Monte Carlo solution can effectively replace the boundary value. The solution to the
optimization problem \eqref{opt} projects the Monte Carlo solution
${\bm v}$ to the null space of ${\bm A}$. The projection works as a ``smoother''
that not only dramatically removes the error term from the Monte Carlo
approximation, but also pushes most error terms to the boundary of the
domain. See the proof and discussion in \cite{dobson2019efficient} for
details.

\subsection{An alternative optimization problem.} \label{Sec:Analysis}
To use artificial neural network approximation, we need to convert
the optimization problem in equation \eqref{opt} to an unconstrained
optimization problem. If we use the penalty
method with penalty parameter $1$ for \eqref{opt}, we have a new
optimization problem
\begin{equation}\label{loss_disc}
\mathop{\min}\limits_{\boldsymbol{u}}\quad  \|\boldsymbol{A} \boldsymbol{u}\|_{2}^{2} + \| \boldsymbol{u} - \boldsymbol{v} \|_{2}^{2},
\end{equation}
where $\boldsymbol{A}$ and $\boldsymbol{v}$ are the same as in equation
\eqref{opt}. We claim that the new optimization problem
\eqref{loss_disc} has a similar effect as the original one in
\eqref{opt}. 

To compare the result, we choose the numerical solution obtained by
the finite difference method, denoted by ${\bm u}^{*}$, as the
baseline, because we have ${\bm A} {\bm u}^{*} = {\bm 0}$. See Appendix
\ref{LAproof} for a more precise description of ${\bm u}^{*}$. Let $\bar{\bm u}$ be the
minimizer of the optimization problem \eqref{loss_disc}. Denote the
error terms of the Monte Carlo simulation and the optimizer by ${\bm e} = {\bm v} - {\bm u}^{*}$ and ${\bm
  z} = \bar{\bm u} - {\bm u}^{*}$ respectively. Let ${\bm A} = {\bm
  A}_{h}$ where $h$ is the grid size of discretization. We make the following assumptions to conduct the
convergence analysis.

\begin{enumerate}[label=\textbf{(A\arabic*)}]
\item\label{Assumption A1} Random vector $\boldsymbol{e}$ has i.i.d. entries whose identical expectation and variance
  are $0$ and $\zeta^2$ respectively. 
\item\label{Assumption A2} Let $\lambda^{h}_{1}, \cdots,
  \lambda^{h}_{r}$ be all nonzero eigenvalues of
  $\boldsymbol{A}_{h}^T\boldsymbol{A}_{h}$. Let $Q(h) = h^{n} \sum_{i = 1}^{r} \left(\frac{1}{1 + h^{-4}
  \lambda^{h}_{i}}\right)^{2} $. We have $Q(h) \rightarrow 0$ as $h
\rightarrow 0$.
\end{enumerate}

%Similar as in \cite{dobson2019efficient}, we assume 

\begin{thm}
  \label{convergence}
  If \ref{Assumption A1} and \ref{Assumption A2} hold, then
$$
 \lim_{h \rightarrow 0} \frac{ \mathbb{E}[\| {\bm z} \|^{2}]}{ \mathbb{E}[\| {\bm e} \|^{2}]} = 0\,.
$$
\end{thm}

{\it Remark: }One needs to multiply the volume of $n$-dimensional grid box when
calculating the discrete $L^{2}$ error. Hence the discrete $L^{2}$ error of ${\bm v}$ is $h^{n/2}
\mathbb{E}[\| {\bm e}\|] = \mathrm{const} \cdot \zeta$. Theorem
\ref{convergence} implies that the error of $\bar{{\bm u}}$ converges
to zero as $h \rightarrow 0$.

Assumption \ref{Assumption A1} assumes the error term ${\bm
      e}$ has i.i.d entries. This is because the error terms of Monte Carlo solutions have very
    little spatial correlation. See Figure \ref{fig1} bottom left
    panel as an example of the spatial distribution of the error term
    of a Monte Carlo solution. The real Monte Carlo simulation has
    smaller error than that in Assumption \ref{Assumption
      A1}, as the absolute error is smaller in the low density area. Assumption \ref{Assumption
      A2}  is due to technical reasons. See Appendix \ref{LAproof} for
    more discussions.

\section{Neural network train algorithms}\label{Sec:Algorithm}
In Theorem \ref{convergence}, we show that the solution to the
unconstrained optimization problem \eqref{loss_disc} converges to the true solution of the
Fokker-Planck equation. Since it is very difficult to do spatial
discretization in high dimension, it is natural to consider the mesh-free
version of the optimization problem \eqref{loss_disc}, in which the
variable $u$ is represented by an artificial neural network.

\subsection{Loss function.} \label{Sec:Loss}
Now let
${\bm \tilde{u}}( {\bm x}, {\bm \theta})$ be an approximation of ${\bm u}$ that is
represented by an artificial neural network with parameter ${\bm
  \theta}$. Inspired by equation \eqref{loss_disc}, we work on the squared error loss function 
\begin{equation}\label{loss}
L(\boldsymbol{\theta}) = \frac{1}{N^{X}}\sum_{i = 1}^{N^{X}}(\mathcal{L}{\bm
  \tilde{u}}(\boldsymbol{x}_i,\boldsymbol{\theta}))^2 +
\frac{1}{N^{Y}}\sum_{j = 1}^{N^{Y}}({\bm
  \tilde{u}}(\boldsymbol{y}_j,\boldsymbol{\theta})-v(\boldsymbol{y}_j))^2
:= L_{1}( {\bm \theta}) + L_{2}({\bm \theta}),
\end{equation}
with respect to $\boldsymbol{\theta}$, where
$\boldsymbol{x}_i\in\mathbb{R}^n, i = 1, 2, \dots, N^{X}$ and
$\boldsymbol{y}_j\in\mathbb{R}^n, j =1, 2, \dots, N^{Y}$ are collocation
points sampled from $D$, and $v(\boldsymbol{y}_j)$ is the Monte Carlo
approximation from a numerical simulation of \eqref{SDE} at ${\bm y}_j$. This loss function \eqref{loss} is in fact the Monte Carlo integration of the following functional
\begin{equation}\label{loss_cont}
J(u) = \|\mathcal{L}u\|^{2}_{L^2(D)} + \|u-v\|^{2}_{L^2(D)},
\end{equation}
which can be seen as the continuous version of the discrete
optimization problem \eqref{loss_disc}.

The loss function \eqref{loss} has two parts. The minimization of
$L_{1}( {\bm \theta})$ is to generate
parameters $\boldsymbol{\theta}^{*}$ that guides the neural network
representation ${\bm \tilde{u}}(\boldsymbol{x},\boldsymbol{\theta}^{*})$ to
fit the Fokker-Planck differential equation $\mathcal{L}{\bm \tilde{u}}=0$
empirically at the training points $\boldsymbol{x}_i, i=1,2,\dots,N^{X}$
(we use automatic differentiation here to generate derivatives of
${\bm \tilde{u}}$ with respect to $\boldsymbol{x}$ using the same parameters
$\boldsymbol{\theta}$). It works as a regularization mechanism such
that the resultant neural network representation 
${\bm \tilde{u}}(\boldsymbol{x},\boldsymbol{\theta}^{*})$ approximates one of the infinitely
many solutions of
the Fokker-Planck equation without boundary conditions. 
Similar to \cite{li2018data} and \cite{dobson2019efficient}, the
second part $L_{2}( {\bm \theta})$ of the loss
function serves as a reference for the solution. It is the low accuracy
Monte Carlo approximation that  guides the neural network
training process to converge to the desired solution, namely the one satisfying 
the stationary Fokker-Planck equation~\eqref{FPE_unique}. The accuracy of
$v( {\bm y}_{i})$ does not have to be very high.  As shown in
\cite{li2018data, dobson2019efficient} and Section \ref{Sec:Analysis},
the optimization problem removes spatially uncorrelated noise in the
Monte Carlo, so that the minimizer is a good approximation of the exact solution of
the Fokker-Planck equation.

Let $\mathfrak{X} := \{\boldsymbol{x}_i;i=1,2,\dots,N^{X}\}$ and
$\mathfrak{Y} := \{\boldsymbol{y}_j;j=1,2,\dots,N^{Y}\}$ be two training
sets that consists of collocation points. To distinguish them, we call
$\mathfrak{X}$ the ``training set'' and $\mathfrak{Y}$ the ``reference
set''. We find that these two sets do not have to be
very large. In our simulations $N^{X}$ ranges from $10^{4}$ to
$10^{5}$, while $N^{Y}$ ranges from $10^{2}$ to $10^{4}$.  This loss
function can be easily trained in a simple feedforward neural network
architecture. (See Appendix \ref{Detail:NN}.) We remark that
the choice of loss function has some similarity to the so called physics-informed
neural network (PINN) studied in \cite{Raissi2019,wu2018physics}. The first
part $\|\mathcal{L}u\|^{2}_{L^2(D)}$ serves a similar role by using the
differential operator from the physics laws there, whereas the second
part of the loss function plays a similar role as the boundary and
initial data in PINN.

The neural network approximation learns the differential operator over
collocation points and
learns the probability density function from the reference data points. It is proved
in many related works that it works effectively to recover a
complicated solution function (see \cite{Raissi2019,Sirignano2018DGM,wu2018physics}). To further accelerate the
training process, we introduce a ``double shuffling'' method that only
uses a small batch of $\mathfrak{X}$ and $\mathfrak{Y}$ in each
iteration to update the parameter. Since $L_{1}$ and $L_{2}$ could have very
different magnitude, in each iteration, we use Adam optimizer~\cite{Kingma2015Adam} to train
$L_{1}$ and $L_{2}$ and update the parameter ${\bm \theta}$ separately (Because Adam optimizer is
invariant to rescaling. See \cite{Kingma2015Adam}.) This method avoids the trouble of rebalancing
the weight of $L_{1}$ and $L_{2}$ during the neural network training. See Algorithm \ref{Neural network training}
for detailed implementation of the ``double shuffling'' method.

\begin{algorithm}[h]
\caption{Neural network training}
\label{Neural network training}
\begin{algorithmic}[1]
\Require
Training set $\mathfrak{X}$ and reference set $\mathfrak{Y}$.
\Ensure
Minimizer $\boldsymbol{\theta}^{*}$ and ${\bm \tilde{u}}(\boldsymbol{x},\boldsymbol{\theta}^{*})$.
\State Initialize a neural network representation ${\bm \tilde{u}}(\boldsymbol{x},\boldsymbol{\theta})$ with undetermined parameters $\boldsymbol{\theta}$.
\State\label{alg1step2}Run Monte Carlo simulation to get an
approximate density $v(\boldsymbol{y}_j)$ at each
reference data points ${\bm y}_{j}, j =1, 2, \dots, N^{Y}$.
\State\label{alg1step3}Pick a mini-batch in $\mathfrak{X}$, calculate the mean
gradient of $L_{1}$, and use the mean gradient to update ${\bm \theta}$
\State\label{alg1step4}Pick a mini-batch in $\mathfrak{Y}$, calculate the mean
gradient of $L_{2}$, and use the mean gradient to update ${\bm
  \theta}$.
\State Repeat steps \ref{alg1step3} and \ref{alg1step4} until the losses of $L_{1}$ and $L_{2}$
are both small enough. 
\State Return $\boldsymbol{\theta}^{*}$ and ${\bm \tilde{u}}(\boldsymbol{x},\boldsymbol{\theta}^{*})$.
\end{algorithmic}
\end{algorithm}

\subsection{Sampling collocation points and reference data.} \label{Sec:Collocation}For many stochastic
dynamical systems \eqref{SDE}, the invariant probability measure is 
concentrated near some small regions or low dimensional manifolds, while the
probability density function is close to zero far away. Hence samples
of the collocation points in $\mathfrak{X}$
and $\mathfrak{Y}$ must effectively represent the
concentration of the invariant probability density function. The solution is to use the dynamics of the system to choose
representative $\mathfrak{X}$ and $\mathfrak{Y}$. We run a numerical trajectory of the
stochastic differential equation \eqref{SDE}, and pick $\alpha\%$ of
the collocation points $\boldsymbol{x}_j$ and ${\bm y}_{j}$ from this
trajectory. When sampling from the long trajectory, we set up an
``internal burn-in time'' $s_{0}$ and only sample at time $\{ n s_{0}
\}_{n = 1, 2, \cdots}$ to avoid samples being too close to each
other. Then
to represent the complement set so that the network can learn small
values from it, we sample the other $1-\alpha\%$ of the collocation points
$\boldsymbol{x}_j$ and ${\bm y}_{j}$ from the uniform distribution on $D$. Since the
concentration part preserves more information of the invariant
distribution density, we usually set $\alpha=50\sim90$. See Algorithm
\ref{Data collocation sampling} for the full detail. 

\begin{algorithm}[h]
\caption{Data collocation sampling}
\label{Data collocation sampling}
\begin{algorithmic}[1]
\Require
Rate $\alpha\in[0.5,0.9]$.
\Ensure
Training collocation points $\boldsymbol{x}_i, i = 1, 2, \dots, N^{X}$ (or $\boldsymbol{y}_j, j =1, 2, \dots, N^{Y}$).
\State Initialize $\boldsymbol{X}_{0}$.
\State Run a numerical trajectory of \eqref{SDE} to time $t_{0}$ to
``burn in''.
\State Choose an internal ``burn in'' time $s_{0}$
\For {$i = 1$ to $N^{X}$}
\State Generate a random number $c_i\sim U([0,1])$. 
\If {$c_i\leq\alpha$}
             \State Let $t_{i} = t_{i-1} + s_{0}$
		\State Run the numerical trajectory of \eqref{SDE} up
                to time $t_{i}$
		\State Let $\boldsymbol{x}_i=\boldsymbol{X}_{t_{i}}$.
                \Else
                \State Let $t_{i} = t_{i-1}$.
		\State Generate a random point $\boldsymbol{x}_i\sim U(D)$.
	\EndIf
\EndFor
\State Return $\boldsymbol{x}_i, i = 1, 2, \dots, N^{Y}$.
\end{algorithmic}
\end{algorithm}

It remains to discuss how to sample the probability density $v( {\bm
  y}_{i})$ for ${\bm y}_{i} \in \mathfrak{Y}$. If the dimension is
low, one can sample $v( {\bm y}_{i})$ use grid-based approaches as in
\cite{li2018data}. For higher dimensional problems, some improvements
on sampling techniques are needed. A memory-efficient Monte Carlo
sampling algorithm for higher dimensional problems is discussed in
Appendix \ref{Sec:MCforHD}. For some stochastic differential equations
with conditional linear structure, a conditional Gaussian sampler
developed in \cite{chen2018efficient} can be used. See Appendix \ref{Sec:CG} for the
full detail.

\section{Numerical examples}\label{Sec:Numerical experiments}
In this section, we use three numerical examples with explicit exact
solution to demonstrate several properties of our Fokker-Planck
solver. Then a six dimensional example is used to demonstrate its
performance in higher dimensions.

\subsection{A 2D ring density}\label{Sec:2dring}
Consider a two dimensional stochastic gradient system
\begin{equation}\label{eq:2dring}
\left\{\begin{array}{l}
dX_t = (-4X_t(X_t^2+Y_t^2-1)+Y_t)\,dt + \sigma\,dW^x_t,\\
dY_t = (-4Y_t(X_t^2+Y_t^2-1)-X_t)\,dt + \sigma\,dW^y_t,\\
\end{array}\right.
\end{equation}
where $W^x_t$ and $W^y_t$ are independent Wiener processes, and we choose diffusion coefficient $\sigma=1$. 
The drift part of equation \eqref{eq:2dring} is a gradient flow of the potential function
$$V(x,y) = (x^2+y^2-1)^2$$
plus a rotation term orthogonal to the equipotential lines of $V$.  Hence the probability density function of the invariant measure of \eqref{eq:2dring} is
$$u(x,y)=\frac{1}{K}e^{-2V(x,y)/\sigma^2},$$
where $K=\pi\int_{-1}^{\infty}e^{-2t^2/\sigma^2}\,dt$ is the
normalization parameter. Note that the orthogonal rotation term does
not change the invariant probability density function. This can be verified by substituting $u(x,y)$ into the Fokker-Planck equation of \eqref{eq:2dring}.

\begin{figure}[htbp]
  {\includegraphics[width=\linewidth]{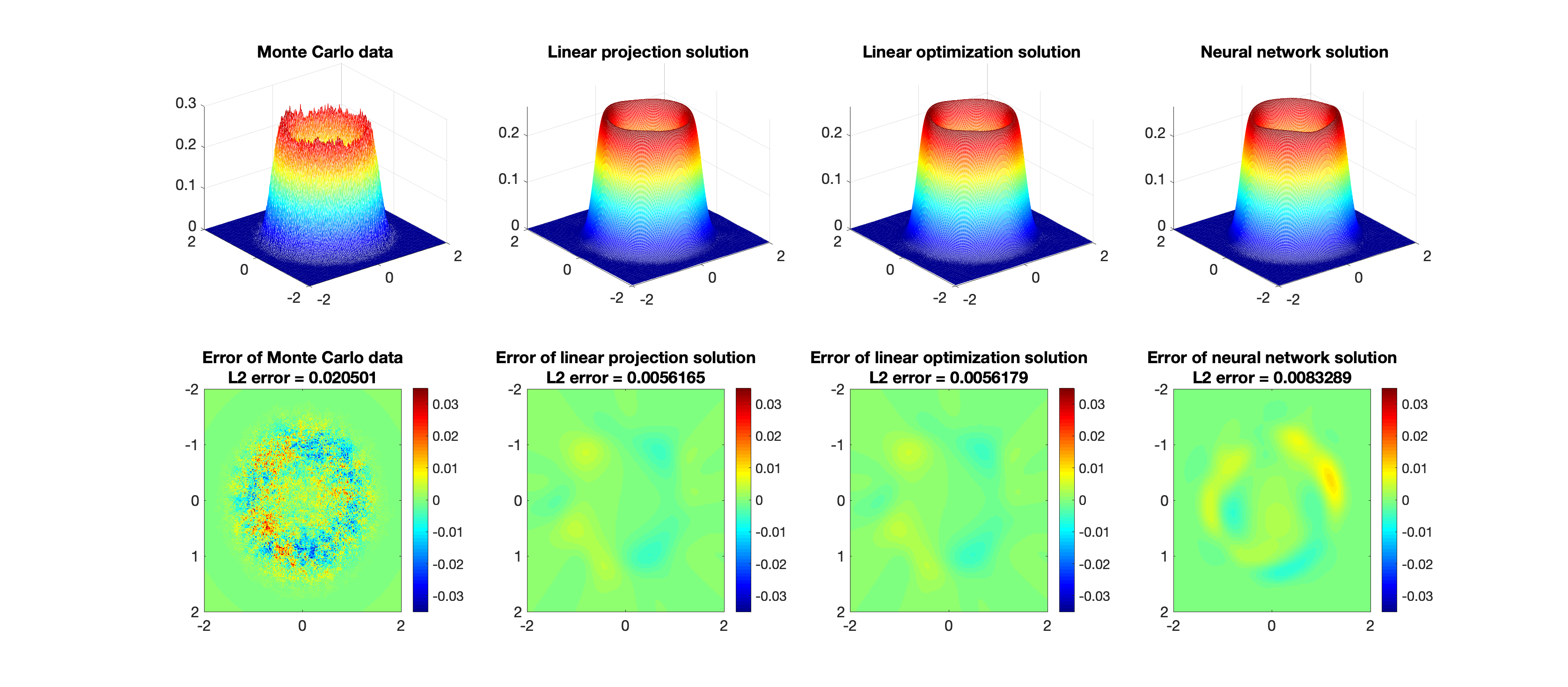}}
\caption{
\label{fig1}A comparison of probability density function of the invariant measure obtained by Monte Carlo simulation, linear
projection method, linear optimization method, and optimization through an artificial neural network. First row: probability density functions. Second row: Distributions of error against the exact solution. 
(Grid size = $200 \times 200$. Sample size of Monte Carlo: $10^{7}$.)
}
\end{figure}

Our first goal  is to compare the performance of
the neural network representation, solution to the constrained
optimization problem, and solution to the unconstrained optimization problem \eqref{loss_disc}. In Figure \ref{fig1}, the first column gives a Monte Carlo
approximation and its error distribution. As expected,
the Monte Carlo approximation is both noisy and inaccurate. Note that the error
term of the Monte Carlo approximation has very little spatial
correlation. This motivates Assumption \ref{Assumption A1}. In the second and the third columns of Figure \ref{fig1} respectively,
we see the data-driven solvers \eqref{opt} and \eqref{loss_disc} can
clearly ``smooth out'' the fluctuation in the Monte Carlo approximation. This confirms the convergence
result proved
in \cite{dobson2019efficient} and Theorem \ref{convergence}. The
last column of Figure \ref{fig1} show that the artificial neural network method with loss
function \eqref{loss} has a similar ``smoothing'' effect. In the neural network training of this example,
we let two sets of collocation points $\mathfrak{X} $ and
$\mathfrak{Y}$ be the set of grid points to compare the result. This
result validates the use of the loss function \eqref{loss} as a
continuous version of the unconstrained optimization problem
\eqref{loss_disc}. We can see
when a grid-based approach is available, it usually has higher accuracy. However, the neural network
method is more applicable to higher dimensional problems.

\medskip

Instead of grid points, the second numerical simulation uses Algorithm 2 with randomly sample collocation
points (using Algorithm \ref{Data collocation sampling}). Figure
\ref{fig2} shows the neural network representations learnt from
various amounts of reference points
$v(\boldsymbol{y}_j)$. The discrete $L^2$ errors is computed with respect to this
grid and demonstrated on the title of each subfigures. The
neural network is then trained with Algorithm 1, in which the norm of
$\mathcal{L}u$ is evaluated at each training point. In order to
numerically check the effect of the Fokker-Planck operator in the loss
function, we train the neural network without calculating
$\mathcal{L}u$, and demonstrate the 
result in Figure \ref{fig6}. More precisely, in Figure \ref{fig6}, we
only use a large training data set $\mathfrak{Y}$. Step 3 in Algorithm
1 is skipped. See Appendix
\ref{Detail:ex1} for details.

In Figure \ref{fig2}, we can see a clear
underfitting when using too few reference points. The training result
becomes satisfactory when the number of training points is $256$ or
larger. As a comparison, if $\mathcal{L}u$ is not added to the loss
function, one needs as large as $16384$ training points to reach the
same accuracy. This shows the advantage of including $\mathcal{L}u$
into the loss function. The differential operator
$\mathcal{L}u$ helps the neural network to find a solution to the
Fokker-Planck equation. And the role of reference points is to make
sure that the Fokker-Planck solution is the one we actually need. We
can train the neural network with only a few hundreds
reference points, and the accuracy of the probability density at those
reference points does not have to be very high.  Similarly to the
discrete case in Theorem \ref{convergence}, the spatially
uncorrelated noise can be effectively removed by training the loss
function $L_{2}$. This
observation is very important in practice, as in high dimension it is
not practical to obtain the probability densities for a very large
reference set, and the result from a high dimension Monte Carlo
simulation is unlikely to be accurate.

\begin{figure}[h]
  {\includegraphics[width = \linewidth]{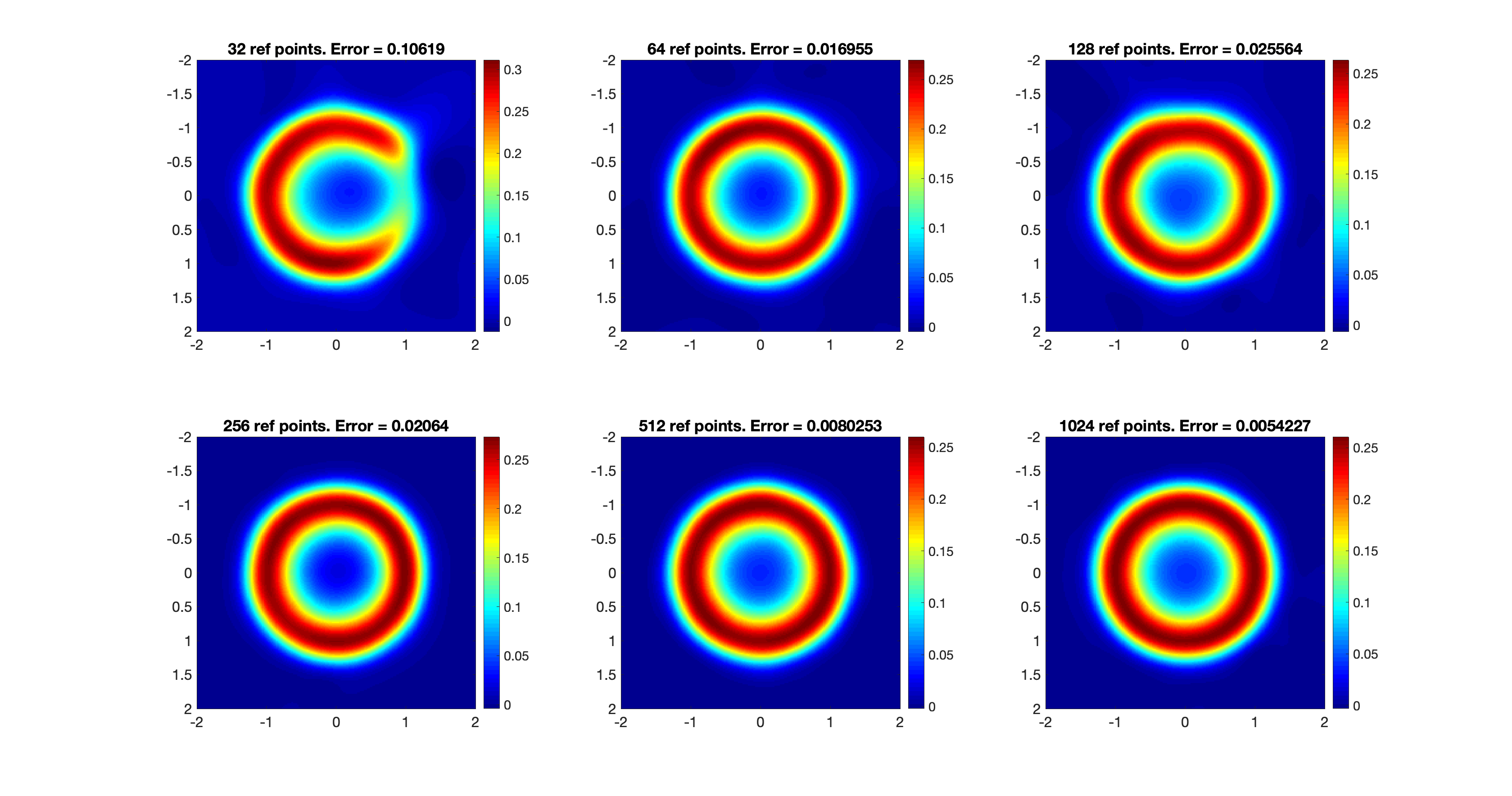}}
  \vskip -0.5cm
\caption{A comparison of different sizes of reference set with $\mathcal{L}u$ being in the loss function. Top left to bottom right: heat map of the invariant probability density function if the ``ring model'' with $32, 64, 128, 256, 512,$ and $1024$ reference points are used. The $L_{2}$ error is shown in the title of each subplot.}
\label{fig2}
\end{figure}

\begin{figure}[htbp]
  {\includegraphics[width = \linewidth]{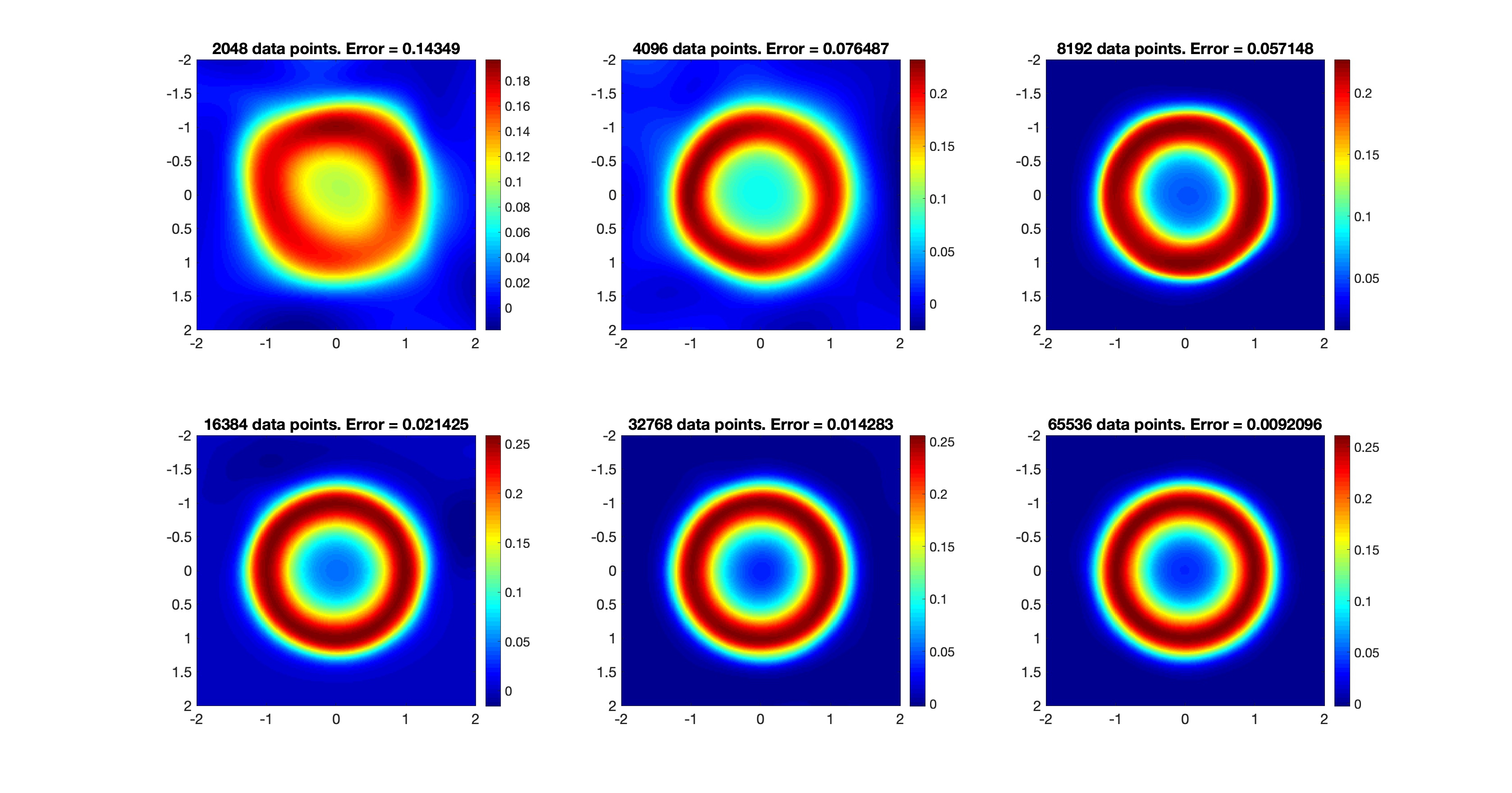}}
  \vskip -0.5cm
\caption{A comparison of different sizes of reference set without $\mathcal{L}u$ being in the loss function. Top left to bottom right: heat map of the invariant probability density function if the ``ring model'' with \num{2048}, \num{4096}, \num{8192}, \num{16384}, \num{32768}, and \num{65536} reference points are used. The $L_{2}$ error is shown in the title of each subplot.}
\label{fig6}
\end{figure}

\subsection{A 2D Gibbs measure}\label{Sec:2dGibbs}
We consider a two dimensional stochastic gradient system
\begin{equation}\label{eq:2dGibbs}
\left\{\begin{array}{l}
dX_t = (X_t^2Y_t-X_t^5)\,dt + \sigma\,dW^x_t,\\
dY_t = (\frac{1}{3}X_t^3-\frac{7}{3}Y_t)\,dt + \sigma\,dW^y_t,
\end{array}\right.
\end{equation}
where $W^x_t$ and $W^y_t$ are independent Wiener processes, and $\sigma=1$ in this example is the strength of the white noise. The drift part of equation \eqref{eq:2dGibbs} is a gradient flow of the potential function
$$V(x,y)=-\frac{1}{3}x^3y+\frac{1}{6}x^6+\frac{7}{6}y^2=\frac{1}{6}(x^3-y)^2+y^2.$$
So the invariant measure of \eqref{eq:2dGibbs} is the Gibbs measure with probability density function
$$u(x,y)=\frac{1}{Z}\exp(-2V(x,y)),$$
where
$Z=\int_{-\infty}^{\infty}\int_{-\infty}^{\infty}\exp(-2V(x,y))\,dxdy$
is the normalization parameter. We choose this system because $Y_{t}$
is conditionally linear with respect to $X_{t}$. We can use this
system to test the conditional Gaussian sampler. 

\begin{figure}[h]
  {\includegraphics[width = 1\linewidth]{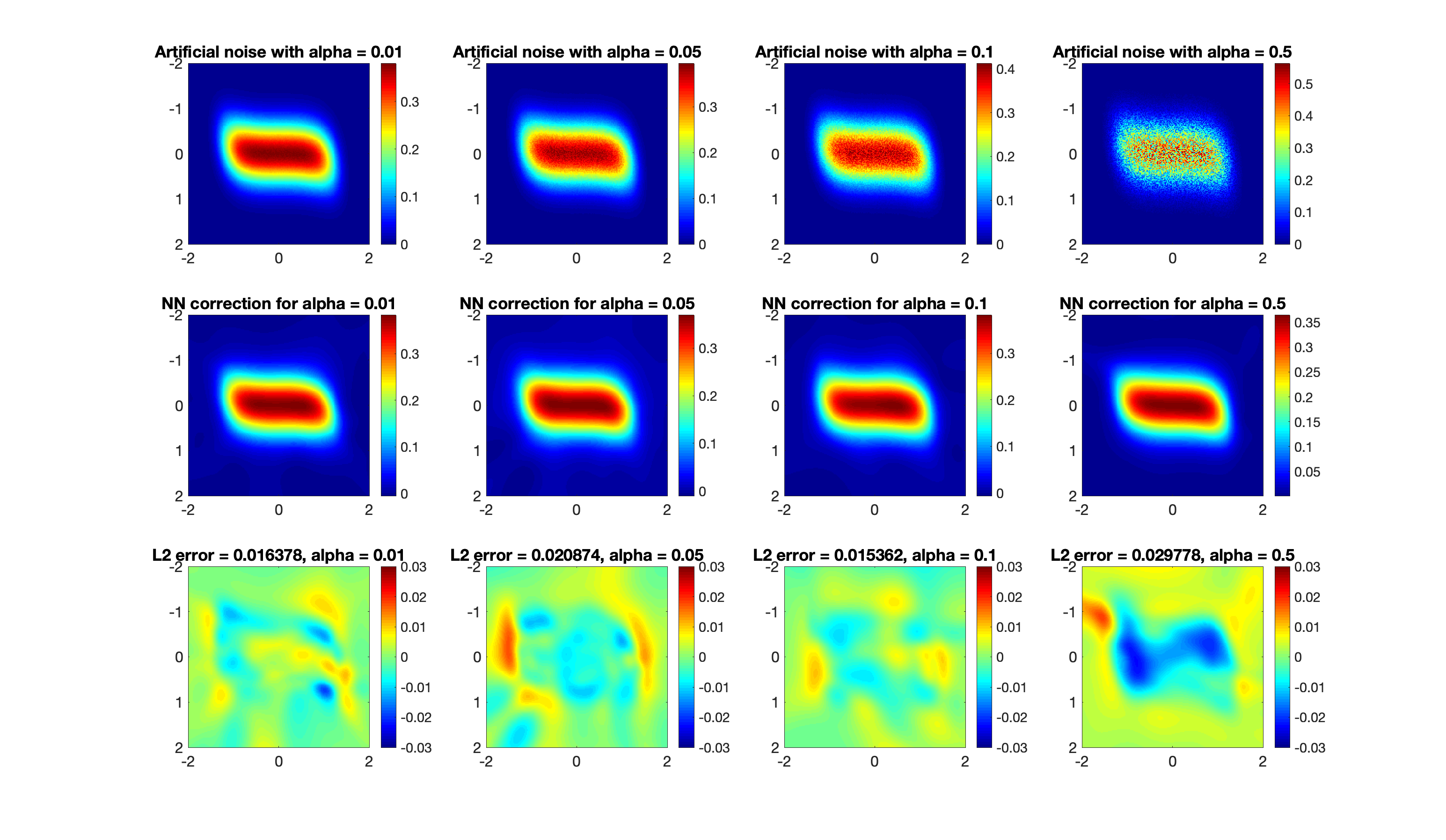}}
  \vskip -0.5cm
\caption{Neural network representations with different level of
  artificial noise. First row: Artificial noises added to the exact
  solution. Second row: Neural network approximation with
  $1024$ reference points and $10000$ training points. Third row: The error of the neural network
  approximation and the
  discrete $L_2$ error.} 
\label{fig:TD1}
\end{figure}

One aim of this numerical experiment is to show the unconstrained
optimization problem used by the artificial neural network can
tolerate spatially uncorrelated noise at a very high level. To demonstrate this, we artificially add a noise to the
exact solution $u(x)$ of the Gibbs measure to get the reference data
$v$. We first run Algorithm \ref{Data collocation sampling} to get
four sets of collocation points $\boldsymbol{y}_j,
j=1,2,\dots,1024$. Then we generate four sets of reference data $v$ at
these collocation points by injecting an artificial noise with maximal
relative error $\alpha$, where $\alpha = 0.01, 0.05, 0.1$, and $0.5$.
Then we run Algorithm \ref{Neural network training} with these sets of reference data $\{v(
{\bm y}_{j})\}_{j = 1}^{1024}$. The first row of Figure
\ref{fig:TD1} shows how the artificial noise is applied by increasing
$\alpha$ and the second row shows the neural network approximation. Observing from the third row of
Figure \ref{fig:TD1}, it is surprising that even when the magnitude of
the multiplicative noise is increased to $0.5$, namely, the relative
error of the Monte Carlo approximation is $50 \%$, the correction
$\tilde{\bm u}(\cdot, \theta)$ is still quite accurate. This shows our method has high tolerance to spatially
uncorrelated noise, which is usually the case of the reference data
obtained from Monte Carlo simulations.

\begin{figure}[h]
  {\includegraphics[width = \linewidth]{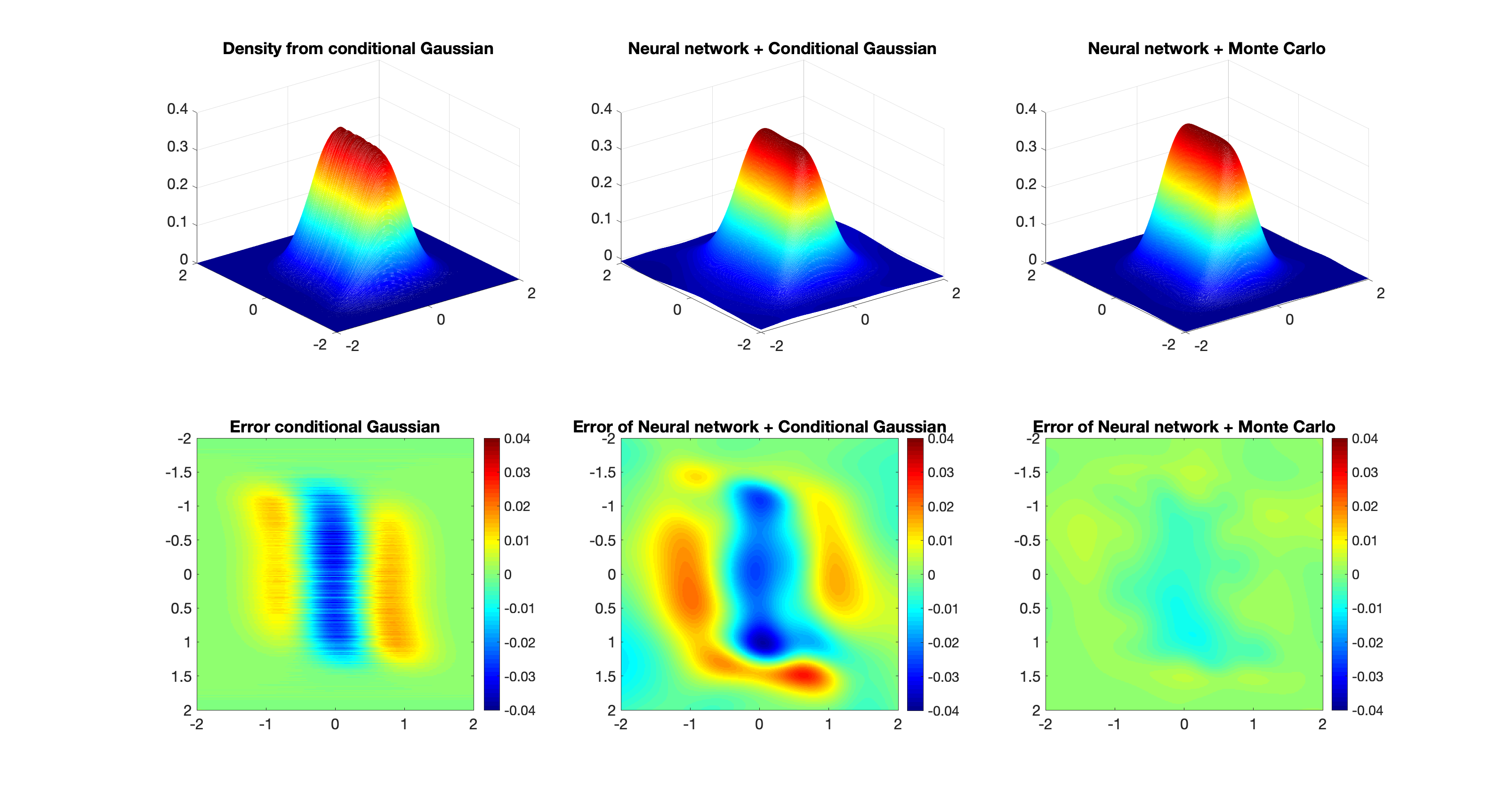}}
  \vskip -0.5cm
\caption{A comparison of the conditional Gaussian method and the
  direct Monte Carlo method. First column: The invariant probability
  density function and its error obtained by the
  conditional Gaussian method. Second column: Top: The invariant probability
  density function and its error obtained by the neural network
  approximation, with $1024$ training points whose densities are
  obtained by Algorithm \ref{Alg:CG}. Bottom: The invariant
  probability 
  density function and its error obtained by the neural network
  approximation, with $1024$ training points whose densities are
  obtained by Monte Carlo simulation. } 
\label{fig3}
\end{figure}

Then we use the conditional Gaussian sampler (Algorithm \ref{Alg:CG}
in Appendix \ref{Sec:CG}) to generate the probability density function. In Figure
\ref{fig3}, we can see there is a small but systematic bias in the
probability density function given by Algorithm \ref{Alg:CG}. We
suspect that this bias comes from the use of one long trajectory in
Algorithm \ref{Alg:CG}. As a result, if we
use it to generate reference data 
$v(\boldsymbol{y}_j)$ for ${\bm y}_{j} \in \mathfrak{Y}$, the error
will be systematic, which is very different from the spatially
uncorrelated noise seen in the Monte Carlo result. This systematic bias makes
the differential operator $\mathcal{L}u$ in the loss function hard to
guide the training, because there are infinitely many functions that
solve $\mathcal{L}u = 0$. To maintain a minimization of the two parts of the
loss function \eqref{loss}, a balance between them forces the neural
network approximation to produce an approximation biased from the exact
solution. In other words, $\boldsymbol{e}=\boldsymbol{v}-\boldsymbol{u}^{*}$ as defined in
Section \ref{Sec:Analysis} for this conditional Gaussian approximation
does not satisfy Assumption \ref{Assumption A1}. Consequently, the
convergence of $\mathbb{E}[\boldsymbol{z}] = \mathbb{E}[ \bar{\bm u} - {\bm
  u}^{*}]$ is not guaranteed. However, the conditional Gaussian sampler
has its advantage in higher dimensions. See Section \ref{Sec:6D} for
more discussion.

%\subsection{Lorenz63 system}
\subsection{A 4D ring density}
Consider a generalization of the stochastic gradient system in Subsection \ref{Sec:2dring} in four dimensional state space
\begin{equation}\label{eq:4dring}
\left\{\begin{array}{l}
dX_t = (-4X_t(X_t^2+Y_t^2+Z_t^2+U_t^2-1)+Y_t)\,dt + \sigma\,dW^x_t,\\
dY_t = (-4Y_t(X_t^2+Y_t^2+Z_t^2+U_t^2-1)-X_t)\,dt + \sigma\,dW^y_t,\\
dZ_t = (-4Z_t(X_t^2+Y_t^2+Z_t^2+U_t^2-1))\,dt + \sigma\,dW^z_t,\\
dU_t = (-4U_t(X_t^2+Y_t^2+Z_t^2+U_t^2-1))\,dt + \sigma\,dW^u_t,
\end{array}\right.
\end{equation}
where $W^x_t, W^y_t, W^z_t$ and $W^u_t$ are independent Wiener
processes, and $\sigma=1$ in this example is the strength of the white
noise. The drift part of equation \eqref{eq:4dring} is a gradient flow
of the potential function 
$$V(x,y) = (x^2+y^2+z^2+u^2-1)^2$$
plus a rotation term orthogonal to the equipotential lines of $V$ in
the first two dimensions of variables $x$ and $y$.  Hence the
invariant measure of \eqref{eq:4dring} is  
$$u(x,y,z,u)=\frac{1}{Z}\exp(-2V(x,y,z,u)),$$
where
$Z=\int_{-\infty}^{\infty}\int_{-\infty}^{\infty}\int_{-\infty}^{\infty}\int_{-\infty}^{\infty}\exp(-2V(x,y,z,u))\,dxdydzdu$
is the normalization parameter. Similar to Subsection
\ref{Sec:2dring}, the rotation term does not change the invariant probability density function,
which can be verified by substituting $u(x,y,z,u)$ into the
Fokker-Planck equation of \eqref{eq:4dring}. 

\begin{figure}[h]
  {\includegraphics[width = 1\linewidth]{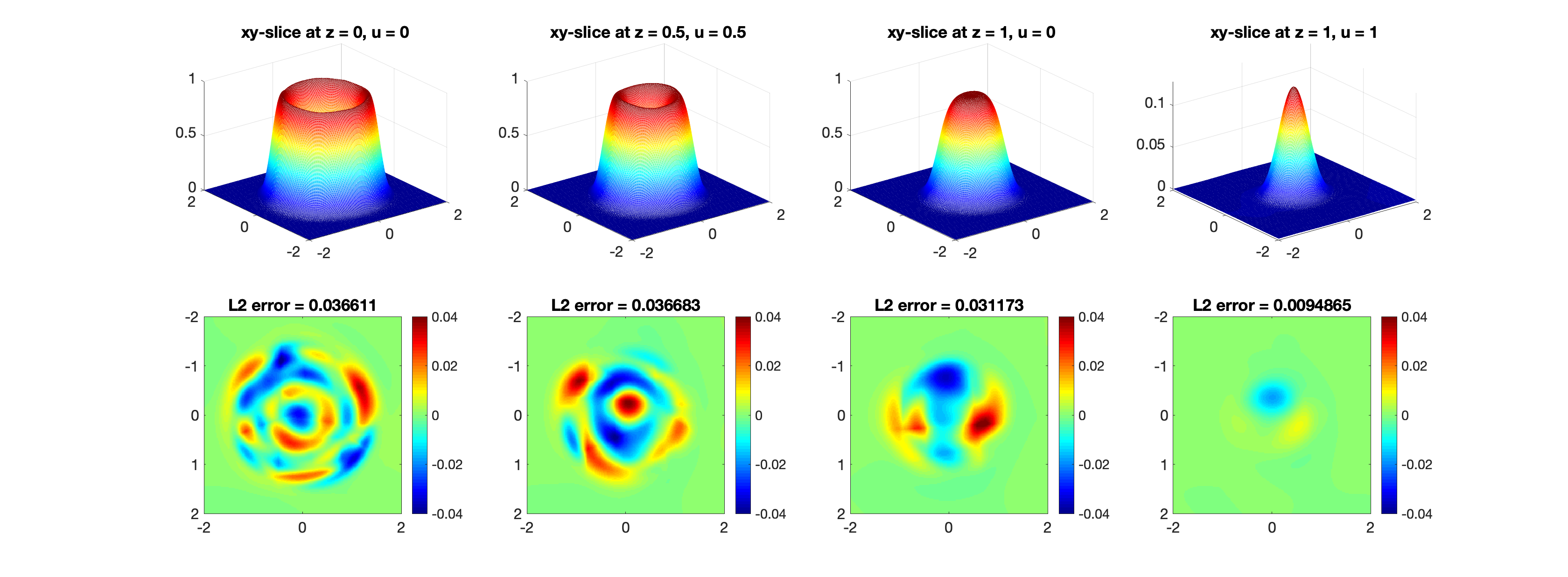}}
  \vskip -0.5cm
\caption{Invariant probability density function of the 4D ring
  (Equation \eqref{eq:4dring}). Total number of reference points is
  \num{10000}. Probability density at reference points is obtained by
  direct Monte Carlo method with $10^{10}$ sample points. First row:
  Invariant probability density functions restricted on the $x$-$y$
  slices with $z = 0, u = 0$. $z = 0.5, u = 0.5$, $z = 1, u = 0$, and $z = 1, u = 1$. Second row: Error of probability density functions when comparing with the exact solution. The discrete $L_{2}$ error is shown in the title of each subplot.}
\label{fig4}
\end{figure}

The aim of this example is to demonstrate the accuracy of neural
network representation in 4D. The numerical domain is
$D=[-2,2]^{2}$. We use Algorithm \ref{Data collocation
  sampling} to sample $10^{4}$ reference points and $10^{5}$ training
points. Probability densities at training points are obtained by
Algorithm \ref{Alg:MCforHD}. After we get the neural network approximation by Algorithm \ref{Neural
  network training}, we evaluate it on the four $x$-$y$ slices for $(z,u)=(0,0), (0.5,0.5), (1,0)$ and $(1,1)$
in Figure \ref{fig4}. Figure \ref{fig4}
also shows the error distributions and the $L^2$ error on these
slices. The errors at all points in
these 4 slices are controlled at a very low level $\leq 0.04$. And the
discrete $L^2$ errors are satisfactory. Figure  \ref{fig4}
illustrates that after training the loss function \eqref{loss} on a
sparse set of reference data, this solution function 
is accurate at any point in $D$. It demonstrates strong
representing power of the neural network approximation, both globally and
locally. We remark that it is not possible to solve this 4D
Fokker-Planck equation with traditional numerical PDE approach. The
divide-and-conquer strategy in \cite{dobson2019efficient} would be
difficult to implement as well, due to the high memory requirement of
a 4D mesh.

\subsection{A 6D conceptual dynamical model for turbulence}\label{Sec:6D}
In this subsection, we consider a six dimensional stochastic dynamical system with conditional Gaussian structure as \eqref{eq:CG1}-\eqref{eq:CG2} with $\boldsymbol{X}^l_{\mathbf{I}}(t)=X_t$ and $\boldsymbol{X}^l_{\mathbf{II}}(t)=(Y^{(1)}_t, Y^{(2)}_t, Y^{(3)}_t, Y^{(4)}_t, Y^{(5)}_t)^T$
\begin{equation}\label{eq:6dCG}
\left\{\begin{array}{l}
dX_t = (-0.1X_t+0.5+0.25X_t(Y^{(1)}_t+Y^{(2)}_t+Y^{(3)}_t+Y^{(4)}_t+Y^{(5)}_t))\,dt + 2\,dW^x_t,\\
dY^{(1)}_t = (-0.2Y^{(1)}_t-0.25X_t^2)\,dt + 0.5\,dW^{(1)}_t,\\
dY^{(2)}_t = (-0.5Y^{(2)}_t-0.25X_t^2)\,dt + 0.2\,dW^{(2)}_t,\\
dY^{(3)}_t = (-Y^{(3)}_t-0.25X_t^2)\,dt + 0.1\,dW^{(3)}_t,\\
dY^{(4)}_t = (-2Y^{(4)}_t-0.25X_t^2)\,dt + 0.1\,dW^{(4)}_t,\\
dY^{(5)}_t = (-5Y^{(5)}_t-0.25X_t^2)\,dt + 0.1\,dW^{(5)}_t,
\end{array}\right.
\end{equation}
where $W^x_t$ and $W^{(i)}_t, i = 1, 2, \dots, 5$, are independent
Wiener processes. This model has been studied in
\cite{chen2018efficient} as a numerical example. 

\begin{figure}[h]
  {\includegraphics[width = 1\linewidth]{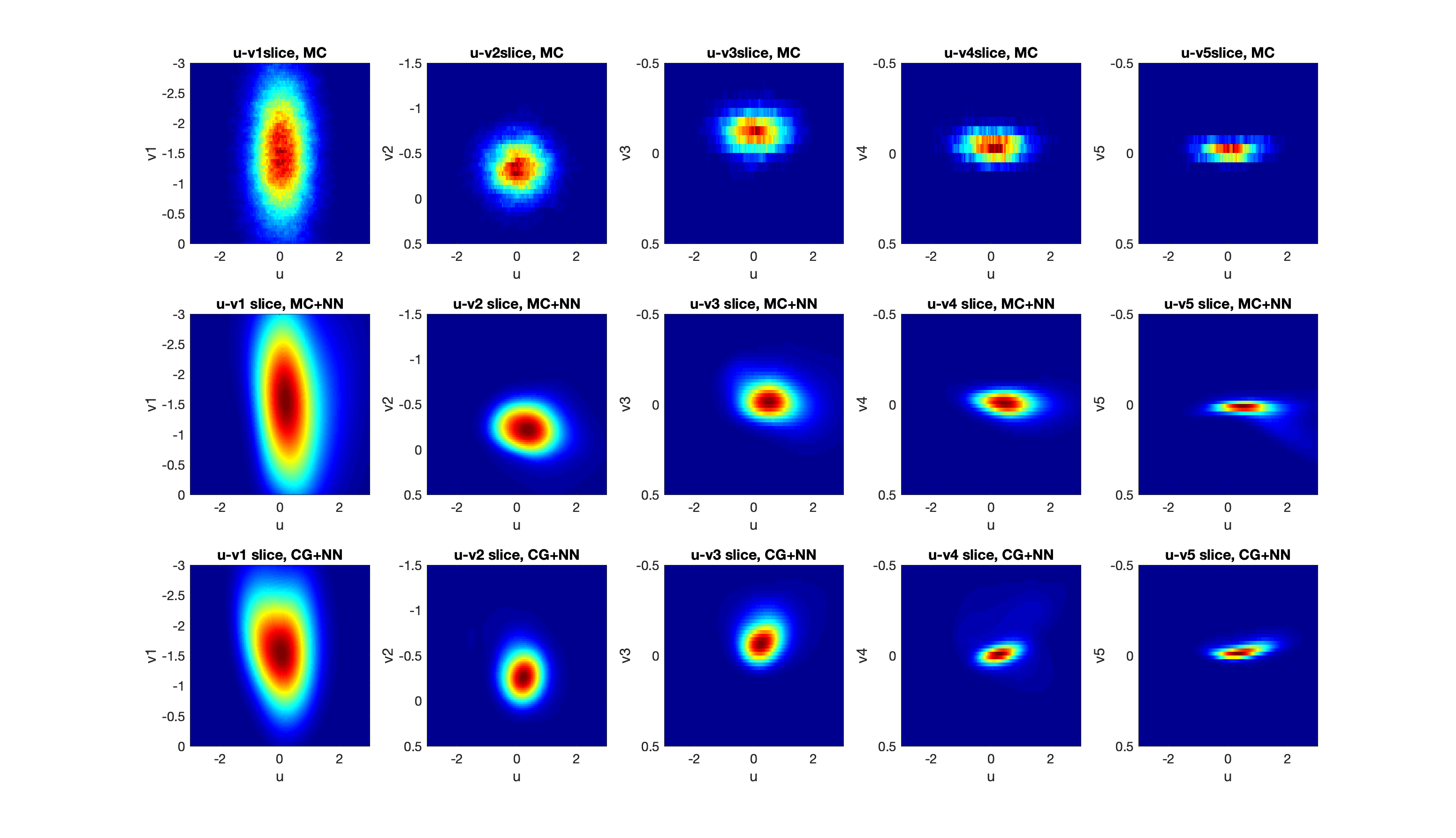}}
  \vskip -0.5cm
\caption{Heat maps of the invariant probability density function of the 6D turbulence model restricted on 2D slices. From left to right: the $i$th plot is the $u$-$v_i, i=1,2,\dots,5$ slice with $v_j=0,j\neq i$.
Top row are invariant probability density functions obtained by the direct
Monte Carlo method. Middle row and bottom row are the
neural network approximation using  Algorithm \ref{Neural network
  training} with probability densities $\{ v({\bm y}_{i})\}$ obtained by
Monte Carlo simulation and the conditional Gaussian sampler, respectively. 
}
\label{fig5}
\end{figure}

In this high dimensional system, we compare the direct Monte Carlo approximation and neural network
approximation with the reference data obtained from both Monte Carlo and the modified
conditional Gaussian sampler in Algorithm \ref{Alg:CG}. It is not
possible to visualize a 6D probability density function, so we compare
probability densities on the central
slices in the 6D state space, namely, the $u$-$v_i, i=1,2,\dots,5$
hyperplanes with $v_j=0,j\neq i$. The first row demonstrates the
probability density functions at the five slices obtained by a direct
Monte Carlo simulation. The solution has low resolution and low
accuracy because It is
difficult to collect enough samples in high dimension. 

 In this example, we
generate a reference point set $\mathfrak{Y}$ with size $N^{Y} =
20000$ using Algorithm \ref{Data
  collocation sampling}. These
collocations are very sparse in this six dimensional region $D$. Then
we use both Monte Carlo approximation and the conditional Gaussian
sampler in Algorithm \ref{Alg:CG} to generate the probability density
$v( {\bm y}_{i})$ for ${\bm y}_{i} \in \mathfrak{Y}$. Note that the simulation time
of Monte Carlo sampler is about $100$ times more than the conditional
Gaussian sampler. Then in both
cases, we use Algorithm \ref{Neural network training} to obtain 
a neural network approximation of the invariant probability measure. After
the neural network is trained in the whole region $D$, we evaluate and
plot it on the five central slices (see the second and third row in Figure
\ref{fig5}). Although a closed-form solution for this example is not
possible, we can still see that the solution obtained by three
different approaches are not very far away from each other. This
confirms the validity of the solutions. The neural
network has low demand ($20000$ points) on reference data points and
fast training speed (less than one hour). After the training, we
can use it to predict the invariant probability density at
any point in the domain. This is a remarkable result, because it is
impossible to solve such as six dimensional problem by using traditional approaches.   

\section{Conclusion and Prospective Works}
We proposed a neural network approximation method for solving the
Fokker-Planck equations. The motivation is that the data-driven method
studied in \cite{li2018data, dobson2019efficient} can be converted to
a similar unconstrained optimization problem, and a mesh-free neural
network solver can be used to solve the ``continuous version'' of this unconstrained optimization
problem. We only present the case of the stationary Fokker-Planck
equation that describes the invariant probability measure, because the
case of the time-dependent Fokker-Planck equation is analogous. By introducing the differential operator of the
Fokker-Planck equation into the loss function, the demand for large training data in
the learning process is significantly reduced. Our simulation shows
that the neural network can tolerate very high noise in the 
training data so long as it is spatially uncorrelated. We believe
this work provides an effective numerical approach to study many high dimensional
stochastic dynamics. 
%This can further help people to understand the
%dynamics and carry out more rigorous result.   

In this paper, the convergence of minimizer of the new unconstrained optimization
problem is only carried out for the discrete case. It is tempting to
extend this result to the space of functions. The problem becomes
trivial and not interesting if we work with the space of $C^{\infty}$
functions and assume that the error term of the Monte Carlo simulation
is a spatial white noise. So we need to consider the more realistic
case such as the Barron space, and find a more realistic assumption to
describe the reference data obtained by the Monte Carlo
simulation. We plan to carry out this study in our subsequent work. In the future, we will also apply this method to more complicated but interesting
systems such as systems with non-Gaussian L\'{e}vy noises and
quasi-stationary distributions.

\bibliographystyle{amsplain}
\bibliography{myref}

\appendix

\section{Proof of Theorem \ref{convergence}}
\label{LAproof}

We first need to describe how the ``baseline'' solution ${\bm u^{*}}$
is obtained. Let $\hat{\bm u} \in \mathbb{R}^{N^{n}}$ denote 
the true solution of the Fokker-Planck equation restricted on the
rectangular grid $\{ x_{i}\}_{i = 1}^{N^{n}}$ in $D$, that is,
$\hat{\bm u}_{i}$ is the solution to equation \eqref{FPE} at point
$x_{i}$. Then ${\bm u}^{*}$ is the numerical solution obtained by the
finite difference method such that ${\bm A} {\bm u}^{*} = 0$ and ${\bm
  u}^{*} = \hat{\bm u}$ at all grid points on $\partial D$. More
precisely, we need to solve a new linear system
$$
\begin{bmatrix}
{\bm A}\\P_{0}
\end{bmatrix}
{\bm u}^{*} = 
\begin{bmatrix}
{\bm 0}\\P_{0} \hat{\bm u}
\end{bmatrix} \,,
$$
where ${\bm A}$ is the aforementioned $(N-2)^{n} \times N^{n}$ matrix, $P_{0}$ is an $( N^{n} - (N-2)^{n})
\times N^{n}$ matrix such that $P_{0} {\bm u}$ gives entries of ${\bm
  u}$ on grid points on $\partial D$. Since the finite difference
method is convergent for second order elliptic PDEs with given boundary
value, when $h \ll 1$, ${\bm u}^{*}$ is a good approximation of
$\hat{\bm u}$. And the accuracy of ${\bm u}^{*}$ is considerable
higher than the result of Monte Carlo simulations. 

\medskip

\begin{proof}[Proof of Theorem \ref{convergence}]
  Let $F( {\bm u}) = {\bm u}^{T} {\bm A}^{T} {\bm A} {\bm u} + (\bm u - \bm v)^T (\bm u - \bm v)$. It is
  easy to see that the minimizer solves
$$
  \frac{\partial F( {\bm u}) }{\partial {\bm u}} = 2 {\bm A}^{T} {\bm
    A} {\bm u} + 2 {\bm u} - 2 {\bm v} = {\bm 0}\,. 
$$
Therefore, the quadratic form \eqref{loss_disc} has a unique minimizer
$\bar{\boldsymbol{u}} =
(\boldsymbol{I}+\boldsymbol{A}^{T}\boldsymbol{A})^{-1}\boldsymbol{v}$.

Denote $\boldsymbol{B} = (\boldsymbol{I}+\boldsymbol{A}^{T}\boldsymbol{A})^{-1}$. Then
$$\boldsymbol{B}^{-1}\boldsymbol{u}^{*} = (\boldsymbol{I}+\boldsymbol{A}^{T}\boldsymbol{A})\boldsymbol{u}^{*} = \boldsymbol{u}^{*},$$
and
$$\boldsymbol{v} = \boldsymbol{B}^{-1}\bar{\boldsymbol{u}} = \boldsymbol{B}^{-1}\bar{\boldsymbol{u}} - \boldsymbol{B}^{-1}\boldsymbol{u}^{*} + \boldsymbol{u}^{*} = \boldsymbol{B}^{-1}\boldsymbol{z} + \boldsymbol{u}^{*}.$$
\iffalse
\begin{align*}
\boldsymbol{v} & = \boldsymbol{B}^{-1}\bar{\boldsymbol{u}}\\
		   & = \boldsymbol{B}^{-1}\bar{\boldsymbol{u}} - \boldsymbol{u}^{*} + \boldsymbol{u}^{*}\\
		   & = \boldsymbol{B}^{-1}\bar{\boldsymbol{u}} - \boldsymbol{B}^{-1}\boldsymbol{u}^{*} + \boldsymbol{u}^{*}\\
		   & = \boldsymbol{B}^{-1}\boldsymbol{z} + \boldsymbol{u}^{*}.
\end{align*}\fi
So $\boldsymbol{z} = \boldsymbol{B}\boldsymbol{e}$. Therefore, $\mathbb{E}[\boldsymbol{z}]=0$ by assumption \ref{Assumption A1}, and $\text{cov}(\boldsymbol{z}, \boldsymbol{z}) = \zeta^2 \boldsymbol{B}\boldsymbol{B}^{T} = \zeta^2 \boldsymbol{B}^2$ by the symmetry of $\boldsymbol{B}$. Furthermore, 
$$\mathbb{E}[\|\boldsymbol{z}\|^2_2] = \text{Trace}(\text{cov}(\boldsymbol{z}, \boldsymbol{z})) = \zeta^2\text{Trace}(\boldsymbol{B}^2) = \zeta^2\sum_{k=1}^{M}\lambda_k^2,$$
where $M=N^n$ and $\{\lambda_k\}_{k=1}^{M}$ are the
eigenvalues of $\boldsymbol{B}$.

For the sake of simplicity let $M = N^{n}$. Recall that ${\bm A}_{h} = h^{2}{\bm A}$. Since
$\boldsymbol{A}_h^T\boldsymbol{A}_h$ is symmetric, there is a
orthonormal matrix $\boldsymbol{Q}$ such that
$\boldsymbol{A}_h^T\boldsymbol{A}_h=\boldsymbol{Q}\boldsymbol{\Lambda}_h\boldsymbol{Q}^T$,
where
$\boldsymbol{\Lambda}_h=\text{diag}\{\lambda^{h}_1,\cdots\lambda^{h}_M\}$
is the diagonal matrix whose diagonal elements are the eigenvalues of
$\boldsymbol{A}_h^T\boldsymbol{A}_h$.  A short calculation shows
$$
\lambda_k=\frac{1}{1+h^{-4} \lambda^{h}_k} \,.
$$

Since $\boldsymbol{A}_h^T\boldsymbol{A}_h$ is positive semi-definite
and $\boldsymbol{A}_h\in \mathbb{R}^{(N-2)^n\times N^n}$ has full rank, 
$\boldsymbol{A}_h^T\boldsymbol{A}_h$ has $N^n - (N-2)^n$
zero eigenvalues, while $r = (N-2)^{n}$ eigenvalues are positive. So for all
sufficiently small $h > 0$, we have
\begin{align*}
\mathbb{E}[\|\boldsymbol{z}\|^2_2] & = \zeta^2\sum_{k=1}^{M}\lambda_k^2 = \zeta^2\sum_{k=1}^{M} \left(\frac{1}{1+h^{-4} \lambda^{h}_k}\right)^2\\
& = \zeta^2 \left[N^n - (N-2)^n + \sum_{i = 1}^{r}\left(\frac{1}{1+h^{-4} \lambda^{h}_i}\right)^2 \right] \,.
\end{align*}
Since $\mathbb{E}[ \| {\bm e} \|^{2}] = \zeta^{2} N^{n}$ and $N
= O(h^{-1})$, by Assumption \ref{Assumption A2}, we
have
$$
  \frac{\mathbb{E}[\|\boldsymbol{z}\|^2_2]}{ \mathbb{E}[ \| {\bm e}
   \|^{2}]} \leq \frac{N^n - (N-2)^n}{N^n} + O(1) Q(h) \rightarrow 0
$$
as $h \rightarrow 0$.
\end{proof}

It remains to discuss why Assumption \ref{Assumption
      A2}  is a valid assumption. Because ${\bm A}$ is obtained by finite difference method, it
    has the form
    $\boldsymbol{A}=h^{-2}(\boldsymbol{A_0}(h)+h\boldsymbol{A_1}(h))$,
    where ${\bm A_{0}}$ (resp. ${\bm A_{1}}$) is a constant matrix if
    $\sigma$ (resp. $f$) is constant. It is extremely difficult to
    rigorously prove Assumption \ref{Assumption
      A2} when $\sigma$ and $f$ in equation \eqref{SDE} are
    location-dependent. In general, the smallest
    nonzero eigenvalue is $O(h^{4})$, but other eigenvalues are
    significantly larger than $O(h^{4})$. There are $O(h^{-n})$ terms
    in the summation in the definition of $Q(h)$. Hence $Q(h)$
    approaches to zero when $h \ll 1$. In Figure \ref{fig8}, we
    numerically verify Assumption \ref{Assumption
      A2}  for 1D and 2D Fokker-Planck equations. We can see that $Q(h) \rightarrow 0$ as $h
    \rightarrow 0$ in both cases. This numerical result is for $\sigma =
    \mathrm{Id}_{n}$ and $f = 0$. Note that $f$ determines ${\bm
      A_{1}}(h)$, which is only a small perturbation of the matrix
    ${\bm A}(h)$. So we expect Assumption \ref{Assumption
      A2}  to hold for any bounded $f$.

\begin{figure}[htbp]
{\includegraphics[width = 0.9\linewidth]{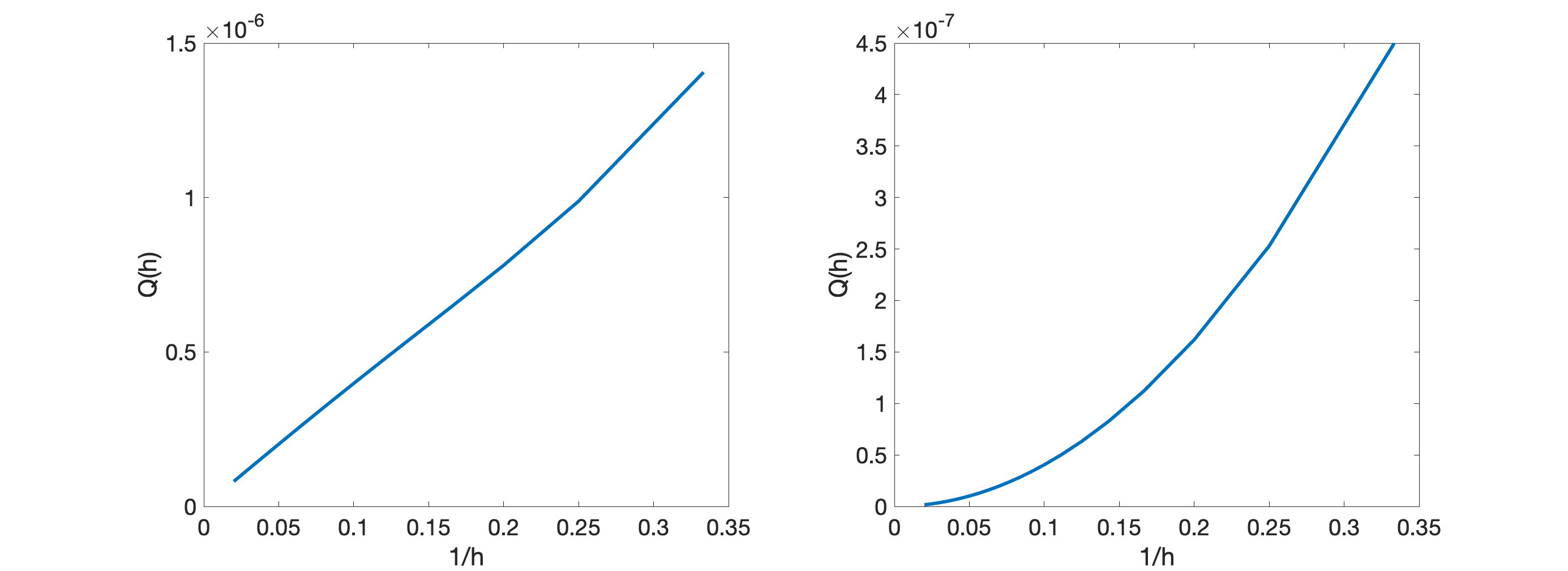}}
\caption{Left: $Q(h)$ vs. $h$ for the discretized 1D
  Fokker-Planck operator. Right: $Q(h)$ vs. $h$ for the discretized 2D
  Fokker-Planck operator.  }
\label{fig8} 
\end{figure}

\section{Sampling probability densities --
  direct Monte Carlo.}\label{Sec:MCforHD}
It remains to discuss the sampling technique to obtain ${\bm
  v(y_{i})}$ for ${\bm y_{i}} \in \mathfrak{Y}$. This step is trivial
when using traditional grid-based method. One only needs to set up a
grid in the numerical domain and count sample points in the
neighborhood of each grid from
a long trajectory. However, instead of the full space, Algorithm
\ref{Neural network training} only requires probability densities at
thousands of reference points ${\bm y_{i}} \in \mathfrak{Y}$. Also a
high dimensional grid could occupy an unrealistic amount of memory. So
we need to improve the efficiency of the Monte Carlo sampler.

As mentioned in Section \ref{Sec: FPE solver}, to obtain a desirable accuracy
of the reference data $v(\boldsymbol{y}_j)$ using a Monte Carlo method,
one needs to run very long numerical trajectories of \eqref{SDE} to
guarantee enough points on the trajectories are counted around
$\boldsymbol{y}_j$. On the other hand, when the dimensionality
increases, the size of reference set $\mathfrak{Y}$ in the training process also
increases. For example, to solve a 6D Fokker-Planck equation, $N^{Y}$
needs to be as large as tens of thousands. If we do not optimize the
algorithm, then every time a new sample point is obtained from a long
trajectory, one must check whether it belongs to the 
neighborhood of each collocation points
$\boldsymbol{y}_j\in\mathfrak{Y}$. This will make the long trajectory sampler too
slow to be useful. 

An alternative approach is to create an $N$ mesh with $N^{n}$ grid
points on $D=\prod_{\iota=1}^{n}[a_\iota,b_\iota]$ and denote the vector of grid
points by $\boldsymbol{y}^{k}, k=1,2,\dots,N^n$. Let $h_{\iota} =
(b_{iota} - a_{iota})/h$ be the grid size. This gives an
$n$-dimensional ``box'' $\prod_{\iota = 1}^{n}[{\bm y}^{k}_{\iota} -
h_{\iota}, {\bm y}^{k}_{\iota} + h_{\iota}]$ around each mesh point
${\bm y}^{k}$. Instead of using Algorithm \ref{Data collocation 
  sampling} to sample $\boldsymbol{y}_j$ directly, we choose the closest mesh point $\boldsymbol{y}^{k_j}$ for
each $\boldsymbol{y}_j$ given by Algorithm \ref{Data collocation
  sampling}. So we have $\mathfrak{Y} = \{ \boldsymbol{y}^{k_j}\}, j =
1, \cdots, N^{Y}$. With the help of the mesh, we can put a sample point
into the corresponding $n$-dimensional ``box'' after implementing $2n$ comparisons. After
running a sufficiently long trajectory, the number of samples in each
box gives the approximate probability density at each grid point. Then
we can look up the probability density of $\boldsymbol{y}^{k_j}$ from
the corresponding boxes. This approach dramatically improves the efficiency, at
the cost of storing a big array with $N^{n}$ points. When $n \geq 4$,
this method could have an unrealistically high demand of the memory.

We propose the following ``splitting'' method to balance the
efficiency and the memory pressure. The idea is to split the dimensions in
to groups to reduce the size of vector stored in the memory. To be specific, we use $n=6$ as
an example to state this method. One can easily generalize it to other
dimensions. For $\mathbb{R}^{6} = \mathbb{R}^{3} \times \mathbb{R}^{3}$, we create an array of
arrays $\mathcal{Q}$ with $2\times N^3$
entries.  The first and second $N^{3}$ entries are for the first and
second $\mathbb{R}^{3}$, respectively. Each entry of $\mathcal{Q}$ is an array of indices of
training points. More precisely, for a collocation point $\boldsymbol{y}_j =
(y^j_1,y^j_2,\dots,y^j_6) \in \mathfrak{Y}$ that is also a mesh point,
we denote its index by
$(n_1,n_2,\dots,n_{6})$, where $n_\iota=
(y^j_\iota-a_\iota)/h_{j} , \iota=1,2,\dots,6$. Then we record the
numbering $j$ in two arrays corresponding to the $(n_1 N^2 + n_2 N + n_3)$-th and the $(N^3 + n_4
N^2 + n_5 N + n_6)$-th entries of $\mathcal{Q}$. When a
sample point $\boldsymbol{x}=(x_1,x_2,\dots,x_6)$ is obtained from the
Monte Carlo sampler, we compute its mesh index
$(n^{\boldsymbol{x}}_1,n^{\boldsymbol{x}}_2,\dots,n^{\boldsymbol{x}}_6)$,
where $n^{\boldsymbol{x}}_\iota=\lfloor (x_\iota-a_\iota)/h_{\iota} + 1/2\rfloor,
\iota=1,2,\dots,6$. Then we check the arrays at the $(n^{\boldsymbol{x}}_1 N^2 +
n^{\boldsymbol{x}}_2 N + n^{\boldsymbol{x}}_3)$-th and the $(N^3 +
n^{\boldsymbol{x}}_4 N^2 + n^{\boldsymbol{x}}_5 N +
n^{\boldsymbol{x}}_6)$-th entries of $\mathcal{A}$. The sample point
${\bm x}$ is associated to the training point ${\bm y}_{j}$ if and only if the
intersection of the two aforementioned arrays is $j$. See
Algorithm \ref{Alg:MCforHD} for the full detail.

\begin{algorithm}[h]
\caption{Reference data sampling with Monte Carlo method for high dimensional spaces $\mathbb{R}^6$}
\label{Alg:MCforHD}
\begin{algorithmic}[1]
\Require
Reference set $\mathfrak{Y} = \{ {\bm y}_{1}, \cdots, {\bm
  y}_{N^{Y}}\}$. ${\bm y}_{i}$ are grid points. 
\Ensure
 Probability densities  $v(\boldsymbol{y}_j)$ at $\boldsymbol{y}_j, j = 1, 2, \dots, N^{Y}$.
\State Set a zero array $\boldsymbol{\eta}$ with length $N^{Y}$ and an
array of arrays $\mathcal{Q}$ that contains $2\times N^3$ empty arrays.
\State Sample $N^{Y}$ collocation points using Algorithm \ref{Data
  collocation sampling} 
\For {$j = 1$ to $N^{Y}$}
	\State Compute $n_\iota=(y^j_\iota-a_\iota)/h_{\iota}, \iota=1,2,\dots,6$.
	\State Add $j$ to the $(n_1 N^2 + n_2 N + n_3)$-th and
        the $(N^3 + n_4 N^2 + n_5 N + n_6)$-th elemental arrays of
        $\mathcal{Q}$. 
\EndFor
\State Initialize $\boldsymbol{X}(0)$ and run a numerical simulation
of equation \eqref{SDE}) for sometime $t_0$ and ``burn in'' time $t_{0}$.
\State Reset $\boldsymbol{X}(0)=\boldsymbol{X}(t_0)$
\For {$l = 1$ to $L$}
	\State Continue the numerical simulation of \eqref{SDE} with
        step size $\Delta t$ to get a new sample point
        $\boldsymbol{x}=\boldsymbol{X}(l\Delta t)$. 
	\State Compute $n^{\boldsymbol{x}}_\iota=\lfloor
        (x_\iota-a_\iota)/h_{\iota} + 1/2 \rfloor , \iota=1,2,\dots,6$.
	\State Check the intersection $\mathcal{B}_{\boldsymbol{x}}$ of the $(n^{\boldsymbol{x}}_1 N^2 + n^{\boldsymbol{x}}_2 N + n^{\boldsymbol{x}}_3)$-th and the $(N^3 + n^{\boldsymbol{x}}_4 N^2 + n^{\boldsymbol{x}}_5 N + n^{\boldsymbol{x}}_6)$-th elemental arrays of $\mathcal{Q}$.
	\If {$\mathcal{B}_{\boldsymbol{x}}=\{j\}$}
		\State $\boldsymbol{\eta}(j)=\boldsymbol{\eta}(j)+1$.
	\EndIf
\EndFor
\State Return $v(\boldsymbol{y}_j)=\boldsymbol{\eta}(j)\prod_{\iota =
  1}^{6} h_{\iota}^{-1}/L, j = 1, 2, \dots, N^{Y}$.
\end{algorithmic}
\end{algorithm}

\section{Sampling probability densities --
  conditional Gaussian sampler.}\label{Sec:CG} As discussed before,
the direct Monte Carlo method still suffers from the curse of
dimensionality. It is more and more difficult to collect enough
samples in a higher dimensional box. To maintain the desired accuracy in high dimensional
spaces, the requirement of sampling grows exponentially with
the dimension. We need to either run much longer numerical trajectories of
\eqref{SDE} or make the grid more coarse. Otherwise the simulation will give a lot
of $v( {\bm y}_{i}) = 0$ at reference points ${\bm y_{i}}$ whose
invariant  probability density is not zero. This makes it not applicable as a reference data in
the neural network training. For some high dimensional problems with conditional linear structure,
the conditional Gaussian framework introduced in
\cite{chen2017beating, chen2018efficient} can be effectively
applied to solve the problem of curse-of-dimensionality. Consider a stochastic differential equation with the
following the conditional linear structure
\begin{align}
d\boldsymbol{X}_{\mathbf{I}} & = [\boldsymbol{A}_0(t,\boldsymbol{X}_{\mathbf{I}})+\boldsymbol{A}_1(t,\boldsymbol{X}_{\mathbf{I}})\boldsymbol{X}_{\mathbf{II}}]\,dt + \boldsymbol{\Sigma}_{\mathbf{I}}(t,\boldsymbol{X}_{\mathbf{I}})\,d\boldsymbol{W}_{\mathbf{I}}(t),\label{eq:CG1}\\
d\boldsymbol{X}_{\mathbf{II}} & = [\boldsymbol{a}_0(t,\boldsymbol{X}_{\mathbf{I}})+\boldsymbol{a}_1(t,\boldsymbol{X}_{\mathbf{I}})\boldsymbol{X}_{\mathbf{II}}]\,dt + \boldsymbol{\Sigma}_{\mathbf{II}}(t,\boldsymbol{X}_{\mathbf{I}})\,d\boldsymbol{W}_{\mathbf{II}}(t),\label{eq:CG2}
\end{align}
where $\boldsymbol{X}(t) = (\boldsymbol{X}_{\mathbf{I}}(t),
\boldsymbol{X}_{\mathbf{II}}(t))\in
\mathbb{R}^{n_{\mathbf{I}}}\times\mathbb{R}^{n_{\mathbf{II}}}$ is the
solution stochastic process. Then given the current path
$\boldsymbol{X}_{\mathbf{I}}(s), s\leq t$, the conditional
distribution of $\boldsymbol{X}_{\mathbf{II}}(t)$ is approximated by a
Gaussian distribution
$$(\boldsymbol{X}_{\mathbf{II}}(t)|\boldsymbol{X}_{\mathbf{I}}(s), s\leq t)\sim N(\olsi{\boldsymbol{X}}_{\mathbf{II}}(t), \boldsymbol{R}_{\mathbf{II}}(t)),$$
where the expectation $\olsi{\boldsymbol{X}}_{\mathbf{II}}(t)$ and
variance $\boldsymbol{R}_{\mathbf{II}}(t)$ follow the ordinary
differential equations 
\begin{align}
d\olsi{\boldsymbol{X}}_{\mathbf{II}} & = [\boldsymbol{a}_0+\boldsymbol{a}_1\olsi{\boldsymbol{X}}_{\mathbf{II}}]\,dt + (\boldsymbol{R}_{\mathbf{II}} \boldsymbol{A}_1^{*}(\boldsymbol{\Sigma}_{\mathbf{I}}\boldsymbol{\Sigma}_{\mathbf{I}}^{*})^{-1}[d\boldsymbol{X}_{\mathbf{I}}-(\boldsymbol{A}_0+\boldsymbol{A}_1\olsi{\boldsymbol{X}}_{\mathbf{II}})\,dt],\label{eq:CGexp}\\
d\boldsymbol{R}_{\mathbf{II}} & = [\boldsymbol{a}_1\boldsymbol{R}_{\mathbf{II}}+\boldsymbol{R}_{\mathbf{II}}\boldsymbol{a}_1^{*} + (\boldsymbol{\Sigma}_{\mathbf{I}}\boldsymbol{\Sigma}_{\mathbf{I}}^{*}) - (\boldsymbol{R}_{\mathbf{II}}\boldsymbol{A}_1^{*}(\boldsymbol{\Sigma}_{\mathbf{I}}\boldsymbol{\Sigma}_{\mathbf{I}}^{*})^{-1}(\boldsymbol{R}_{\mathbf{II}}\boldsymbol{A}_1^{*})^{*}]\,dt.\label{eq:CGcov}
\end{align}

The original algorithm in \cite{chen2018efficient} is for simulating
the time evolution of the probability density function. The
probability density function is obtained by averaging the conditional
probability density of many independent trajectories of equation
\eqref{SDE}. Since the focus of this paper is the invariant probability density
function, we make some modification to the conditional Gaussian
framework in \cite{chen2018efficient}. The main difference is that we
use one long trajectory to simulate the conditional probability
density. This is because the speed of convergence of the evolution of transient distribution to the invariant distribution of equation
\eqref{SDE} is unknown. In a simulation, we don't know when the probability density
function becomes a satisfactory approximation of the invariant probability density function.

In equation \eqref{eq:CG1}-\eqref{eq:CG2}, the first part
$\boldsymbol{X}_{\mathbf{I}}$ is usually in a relatively low dimension
$n_{\mathbf{I}}$. So for this part, a Monte Carlo approximation is
reliable. Let $\mathfrak{Y}$ be the set of reference points. Denote
the two coordinates of a reference point ${\bm y}_{i} \in
\mathfrak{Y}$ by ${\bm y}_{i}^{\mathbf{I}}$ and ${\bm y}_{i}^{\mathbf{II}}$ 
respectively. Then we run a long numerical trajectory
$\boldsymbol{X}$, for \eqref{eq:CG1}-\eqref{eq:CG2} and evaluate the
trajectory at discrete times $0=t_0<t_1<t_2<\cdots<t_I=T$. Denote the
visiting times of $\boldsymbol{X}_{\mathbf{I}}$ to an $h$-neighborhood of
${\bm y}_{i}^{\mathbf{I}}$ by $t_{k_{1}}, \cdots, t_{k_{S(j)}}$. Then at time
$t_{k_{i}}$, the conditional probability density of at ${\bm
  y}_{i}^{\mathbf{II}}$ is
\begin{equation}\label{eq:CGhighpdf}
\begin{array}{ll}
v_{i,j}
& = f_{(\boldsymbol{y}^{\mathbf{II}}_j(t_{k_{i}})|\boldsymbol{y}^{\mathbf{I}}_j(s), s\leq t_{k_{i}})}(\boldsymbol{y}^{\mathbf{II}}_j)\\
& = \frac{1}{\sqrt{(2\pi)^{n_{\mathbf{II}}}|\boldsymbol{R}_{\mathbf{II}}(t_{k_{i}})|}}\exp(-\frac{1}{2}(\boldsymbol{y}^{\mathbf{II}}_j-\olsi{\boldsymbol{X}}_{\mathbf{II}}(t_{k_{i}}))^T \boldsymbol{R}_{\mathbf{II}}(t_{k_{i}})^{-1}(\boldsymbol{y}^{\mathbf{II}}_j-\olsi{\boldsymbol{X}}_{\mathbf{II}}(t_{k_{i}})))
\end{array}
\end{equation}
according to the Gaussian distribution $N(\olsi{\boldsymbol{X}}_{\mathbf{II}}(t_{k_{i}}),
\boldsymbol{R}_{\mathbf{II}}(t_{k_{i}}))$.This gives
a more reliable approximation for the reference data $v(\boldsymbol{y}_j)$
\begin{equation}\label{eq:CGave}
v(\boldsymbol{y}_j)=\frac{1}{S(j)}\sum_{i=1}^{S(j)}v_{i,j}.
\end{equation}
See Algorithm \ref{Alg:CG} for the full
detail. 

\begin{algorithm}[h]
\caption{Reference data sampling for high dimensional spaces with conditional Gaussian structure}
\label{Alg:CG}
\begin{algorithmic}[1]
\Require
Conditional linear stochastic differential equations \eqref{eq:CG1}
and \eqref{eq:CG2}.
\Ensure
Reference approximation $v(\boldsymbol{y}_j)$ at $\boldsymbol{y}_j=(\boldsymbol{y}^{\mathbf{I}}_j, \boldsymbol{y}^{\mathbf{II}}_j), j = 1, 2, \dots, N^{Y}$.
\State Initialize $\boldsymbol{X}(0)$ and run a numerical simulation
of equation \eqref{eq:CG1}-\eqref{eq:CG2} for sometime $t_0$ and ``burn in''.
\State Reset $\boldsymbol{X}(0)=\boldsymbol{X}(t_0)$
\State Continue the numerical simulation of \eqref{SDE} for a relatively large $T$ and collect $\boldsymbol{X}_{\mathbf{I}}(s), s\leq t$.
\State Run a numerical solver for \eqref{eq:CGexp} and \eqref{eq:CGcov} to get $\olsi{\boldsymbol{X}}_{\mathbf{II}}(t)$ and $\boldsymbol{R}_{\mathbf{II}}(t)$.
	\For {$j = 1$ to $M$}
	\State Record $\boldsymbol{X}(t_{k_{i}}), i=1,2,\dots, S(j) \in B(\boldsymbol{y}^{\mathbf{I}}_j, h)$ in $\mathbb{R}^{n_{\mathbf{I}}}$.
	\State Evaluate $v_{k,j}$ using \eqref{eq:CGhighpdf}.
	\EndFor
\State Return $v(\boldsymbol{y}_j), j = 1, 2, \dots, N^{Y}$ using \eqref{eq:CGave}.
\end{algorithmic}
\end{algorithm}

\section{Numerical simulation details}
\subsection{Parameter of the neural network.}
\label{Detail:NN}
Throughout this paper, we use a small feed-forward neural network with $6$ hidden layers, each
of which contains $16, 128, 128, 128, 16, 4$ neurons respectively, to approximate
the solution to Fokker-Planck equation in all numerical
examples. The output layer always has one neuron. Number of neurons in
the input layer depends on the problem. All activation functions are the sigmoid function. We choose
sigmoid function because (1) the solution of the Fokker-Planck equation is
everywhere nonnegative and (2) the second order derivative of the
neural network output is included in the loss function.

\subsection{Numerical example 1.} \label{Detail:ex1} In Figure \ref{fig1},
the Monte Carlo solution is obtained by running an Euler-Maruyama numerical
scheme for \eqref{eq:2dring} with $10^7$ steps and calculating the
empirical probability on a $200\times200$ mesh of the region
$D=[-2,2]\times[-2,2]$. The time step size is $0.001$. Optimization
problems in equation \eqref{opt} and \eqref{loss_disc} are solved by
linear algebra solvers. The neural network training with loss function \ref{loss} uses all probability
densities at grid points obtained by the same Monte Carlo simulation. The architecture of the artificial neural network is
described in Section \ref{Detail:NN} with two input neuron and one
output neuron. 

In Figure \ref{fig2}, Algorithm \ref{Data collocation sampling} is used to generate reference
points and  training points.  The number of
reference points in six panels of Figure \ref{fig2} are $32, 64, 128,
256, 512$, and $1024$ respectively. The number of training points is
$10000$ in all cases. All 
probability densities $v( {\bm y}_{i})$ at training points are
obtained by Algorithm \ref{Alg:MCforHD}, which runs the Euler-Maruyama
scheme for $10^{8}$ steps. Then we train the artificial
neural network with loss function \eqref{loss}. The architecture of the artificial neural network is
described in Section \ref{Detail:NN} with two input neuron and one
output neuron. The trained neural network is evaluated on a $400\times
400$ grid.

When generating Figure \ref{fig6}, we use the loss function without
$\mathcal{L}u$, so there is no training set $\mathfrak{X}$. In six
panel of Figure \ref{fig6}, the numbers of reference points
$v(\boldsymbol{y}_j)$ with probability densities are
\num{2048}, \num{4096}, \num{8192}, \num{16384}, \num{32768}, and
\num{65536}, respectively. Reference points are obtained by Algorithm \ref{Data collocation sampling}. The probability density at each reference
point is exact (obtained from the Gibbs density). The architecture of the artificial neural network is
described in Section \ref{Detail:NN} with two input neuron and one
output neuron. The trained neural network is evaluated on a $400\times
400$ grid.

\subsection{Numerical example 2.} \label{Detail:ex2} In Figure
\ref{fig:TD1}, Algorithm \ref{Data collocation sampling} is used to generate $1024$
reference points and $10000$ training points. Then we artificially inject some noise into the training
data. For a reference point ${\bm y}_{i}
\in \mathfrak{Y}$, we have $v_i(\boldsymbol{y}_j)=r_{i}(\boldsymbol{y}_j)u(\boldsymbol{y}_j)$,
where $u$ is the Gibbs density, $r_{i}\sim U([1-\alpha, 1+\alpha])$ is a random variable
uniformly distributed in a range $[1-\alpha, 1+\alpha]$. Here,
$\alpha$ controls the ``strength'' of the artificial noise. Four
different values $\alpha = 0.01, 0.05, 0.1$ and
$0.5$ are used to generate four different reference data sets $\{v(
{\bm y}_{j})\}_{j = 1}^{1024}$. The architecture of the artificial neural network is
described in Section \ref{Detail:NN} with two input neuron and one
output neuron. The trained neural network is evaluated on a $400\times
400$ grid.

In Figure \ref{fig3}, the conditional Gaussian simulation is
obtained by  running Algorithm \ref{Alg:CG}. The trajectory is
recorded at $50000$ discrete times. The probability density at
$400\times 400$ grid points are evaluated by using Algorithm
\ref{Alg:CG}. (Two panels in the first column).  Next, Algorithm \ref{Data collocation sampling} is used to generate $1024$
reference points and $10000$ training points. The probability
densities at those training points are evaluated by running Algorithm
\ref{Alg:CG} (same sample size as above) and Algorithm
\ref{Alg:MCforHD}, respectively. Two sets of probability densities at
reference points are used in the neural network training (with loss
function \eqref{loss}) to generate subplots in the second and third
column, respectively. The architecture of the artificial neural network is
described in Section \ref{Detail:NN} with two input neuron and one
output neuron. The trained neural network is evaluated on a $400\times
400$ grid.

\subsection{Numerical example 3.} \label{Detail:ex3} In Figure
\ref{fig4}, algorithm \ref{Data collocation sampling} is used to
sample $10000$ reference points and $10^{5}$ training points in the
domain $D = [-2, 2]^{4}$. Then we run algorithm \ref{Alg:MCforHD} with
$10^{10}$ steps of the Euler-Maruyama scheme to estimate the
probability densities at training points. The values of the probability density function are rescaled on the whole domain $D$ such that
the maximum is $1$ (for otherwise the neural network cannot easily
learn the distinction among small values). Then we train the artificial
neural network with loss function \eqref{loss}. The architecture of the artificial neural network is
described in Section \ref{Detail:NN} with four input neuron and one
output neuron. The trained neural network is evaluated at four
$(x,y)$-slices for $(z,u) = (0,0), (0.5, 0.5), (1,0)$ and $(1,1)$
respectively. Each $(x,y)$ slice contains $400\times 400$ grid
points. 

\subsection{Numerical example 4.} \label{Detail:ex4} In Figure
\ref{fig5}, the numerical domain is $[-3, 3] \times [-3, 0] \times
[-1.5, 0.5] \times  [-0.5, 0.5] \times [-0.5, 0.5]\times [-0.5, 0.5]
\subset \mathbb{R}^{6}$. A direct Monte Carlo simulation that uses $8 \times
10^{9}$ steps of Euler-Maruyama scheme is used to generate
subplots in the first row. The grid size
of the Monte Carlo simpler is $0.05$. Then we use Algorithm \ref{Data collocation sampling} to
sample $20000$ reference points and $10^{5}$ training points in the
domain $D$. Two sets of probability densities at reference points are obtained
using two approaches. The first approach uses Algorithm \ref{Alg:MCforHD} with
$4 \times 10^{10}$ steps of the Euler-Maruyama scheme. The second
approach uses $8 \times 10^{5}$ samples from the conditional Gaussian
sampler in Algorithm \ref{Alg:CG}. The values of the probability density function are rescaled on the whole domain $D$ such that
the maximum is $1$. Then we train two artificial
neural networks (with loss function \eqref{loss}) using the same
collocation points but two sets of probability densities at
reference points. The results are shown in the second and third row of
Figure \ref{fig5} respectively. The architecture of the artificial neural network is
described in Section \ref{Detail:NN} with six input neuron and one
output neuron. The trained neural network is evaluated at five
$(u,v_{i})$-slices for $i = 1, \cdots, 5$ centering at the origin. 

\end{document}